\documentclass[11pt]{amsart}

\usepackage{amsmath,amsthm,amscd,amssymb}
\usepackage[letterpaper,left=1.5in,right=1.5in,top=1.25in,bottom=1.25in]{geometry}
\usepackage{pb-diagram}
\usepackage{subfig}
\usepackage{graphicx, xcolor}

\newtheorem{thm}{Theorem}[section]
\newtheorem{lem}[thm]{Lemma}
\newtheorem{prop}[thm]{Proposition}
\newtheorem{cor}[thm]{Corollary}

\theoremstyle{definition}

\newtheorem{defn}[thm]{Definition}

\numberwithin{equation}{section}

\title[Shake Slice and Shake Concordant Links]{Shake Slice and Shake Concordant Links}

\author{Anthony Bosman}
\address{Andrews University\\4260 Administration Dr., 49104 Berrien Springs, USA}
\email{bosman@andrews.edu}


\begin{document}

\date{\today}
\begin{abstract}  
We can construct a $4$-manifold by attaching $2$-handles to a $4$-ball with framing $r$ along the components of a link in the boundary of the $4$-ball. We define a link as {\it $r$-shake slice} if there exists embedded spheres that represent the generators of the second homology of the $4$-manifold. This naturally extends $r$-shake slice, a generalization of slice that has previously only been studied for knots, to links of more than one component. We also define a relative notion of {\it shake $r$-concordance} for links and versions with stricter conditions on the embedded spheres that we call {\it strongly $r$-shake slice} and {\it strongly $r$-shake concordance}. We provide infinite families of links that distinguish concordance, shake concordance, and strong shake concordance. Moreover, for $r=0$ we completely characterize shake slice and shake concordant links in terms of concordance and string link infection. This characterization allows us to prove that the first non-vanishing Milnor $\overline{\mu}$ invariants are invariants of shake concordance. We also argue that shake concordance does not imply link homotopy.
\end{abstract}

\maketitle

\section{Introduction}

Given a knot $K$ we can form a 4-manifold $W_K^r$, for some integer $r$, by attaching a 2-handle to the ball $B^4$ along $K$ with framing $r$. If $K$ is a slice knot, then the core of the 2-handle and the slice disk form a smoothly embedded sphere in $W_K^r$ that represents a generator of the middle-dimensional homology class $H_2(W_K^r)\cong \mathbb{Z}$. In general, we will say that the knot $K$ is {\it $r$-shake slice}, a notion introduced in \cite{akbulut77}, if there is such a smoothly embedded sphere in $W_K^r$ that generates $H_2(W_K^r)$. It is natural to ask if there are $r$-shake slice knots that are not slice. In \cite{akbulut77}, Akbulut provided examples of 1-shake slice and 2-shake slice knots that are not slice. Lickorish provided additional such examples in \cite{lick}. More recently constructions for infinitely families of r-shake slice knots that are not slice have been provided in \cite{akbulut93} and \cite{abejongetc} for all nonzero $r$. It is still unknown if there are $0$-shake slice knots that are not slice.

There is also a relative version: Given knots $K\hookrightarrow S^3\times\{0\}$ and $K'\hookrightarrow S^3\times\{1\}$ and integer $r$ form the 4-manifold $W_{K,K'}^r$ by adding two 2-handles to $S^3\times[0,1]$ along $K$ and $K'$ with framing $r$. Then we call $K$ and $K'$ $r$-shake concordant if there exists a smoothly embedded 2-sphere in $W_{K,K'}^r$ representing the $(1,1)$ class of $H_2(W_{K,K'}^r)$. Cochran and Ray \cite{cochranray16} showed that there exists an infinite family of topologically slice knots that are pairwise $0$-shake concordant but distinct in smooth concordance.

In Section 2 we extend the notion of $r$-shake slice to $m$-component links by considering the 4-manifold $W_L^r$ formed by attaching $m$ 2-handles to $B^4$ along the components of $L$ with framing $r$. We say the link is $r$-shake slice if there exists embedded spheres representing the generators $(1,0,...,0),...,(0,0,...,1)$ of $H_2(W_L^r)\cong \mathbb{Z}^m$. Similarly, we may extend the notion of $r$-shake concordance to $m$-component links. We also introduce stricter versions of these notions which we call {\it strongly $r$-shake slice} and {\it strongly $r$-shake concordance} by restricting how the spheres interact with the 2-handle. Of interest, then, is how concordance, $r$-shake concordance, and strong $r$-shake concordance relate for links of more than one component. They are distinguished by the following results of Section 3:
\newtheorem*{cor:shakenotstrong}{Corollary~\ref{cor:shakenotstrong}}
\begin{cor:shakenotstrong} There exists an infinite family of 2-component links that are pairwise shake concordant, but not pairwise strongly shake concordant.\end{cor:shakenotstrong}

\newtheorem*{prop:strongbutnotconc}{Proposition~\ref{prop:strongbutnotconc}}

\begin{prop:strongbutnotconc} There exists an infinite family of 2-components links with trivial components that are all strongly shake concordant to the Hopf link, but none of which are concordant to the Hopf link. \end{prop:strongbutnotconc}

Cochran and Ray provided in \cite{cochranray16} a complete characterization of $r$-shake concordant knots in terms of concordance and winding number one satellite operators. In Section 4 we extend this characterization to links of more than one component with string link infection generalizing satellite operators:

\newtheorem*{thm:classification}{Theorem~\ref{thm:classification}}
\begin{thm:classification}
The $m$-component links $L$ and $L'$ are shake concordant if and only if there exist links obtained by string link infection $I(S,J,\mathbb{E}_\varphi)$ and $I(S', J',\mathbb{E}'_{\varphi'})$ that are concordant for some:
		\begin{itemize}
				\item $m$-component slice links $S$ and $S'$,
				\item $m$-component string links $J$, $J'$ with closures $\widehat{J}=L$ and $\widehat{J}'=L'$,
				\item and embedded multidisks $\mathbb{E}_\varphi$ that respects $L$ and $\mathbb{E}'_{\varphi'}$ that respects $L'$.
		\end{itemize}
\end{thm:classification}

We say $\mathbb{E}_\varphi$ {\it respects} $L$ if $\mathbb{E}_\varphi$ has $m$ subdisks $D_1,D_2,...,D_m$ and each link component $L_i$ of $L$ intersects the subdisk $D_j$ algebraically once when $i=j$ and algebraically zero times when $i\neq j$.

We may also characterize $r$-shake slice links. This may provide a means of exhibiting a link that is 0-shake slice but not slice; it is not known if any such link exists.

\newtheorem*{cor:classification}{Corollary~\ref{cor:classification}}
\begin{cor:classification}
	The $m$-component link $L$ is shake slice if and only if the link obtained by string link infection $I(S,J,\mathbb{E}_\varphi)$ is slice for some:
		\begin{itemize}
				\item $m$-component slice link $S$,
				\item $m$-component string link $J$ with closure $\widehat{J}=L$,
				\item and embedded multidisk $\mathbb{E}_\varphi$ that respect $S$.
		\end{itemize}
\end{cor:classification}

Finally, in Section 5 we study the Milnor $\overline{\mu}$ invariants, an important and well-studied family of concordance invariants \cite{casson} first introduced in \cite{milnor54} and \cite{milnor58}. For an $m$-component link $L$, recall that the invariant $\overline{\mu}_L(I)$ is defined for each multi-index $I=i_1i_2...i_l$ where $1\leq i_1,...,i_l\leq m$; we say $I=i_1i_2...i_l$ has length $|I|=l$. We show that the first non-vanishing Milnor invariant is an invariant of shake concordance:

\newtheorem*{thm:firstmilnor}{Theorem~\ref{thm:firstmilnor}}
\begin{thm:firstmilnor}
	If two links $L$ and $L'$ are shake concordant, then they have equal first non-vanishing Milnor invariants. That is, if for some multi-index $I$, $\overline{\mu}_L(I)\neq 0$ and $\overline{\mu}_L(J)=0$ for all $|J|<|I|$, then $\overline{\mu}_L(I)=\overline{\mu}_{L'}(I)$ and $\overline{\mu}_{L'}(J)=0$ for all $|J|<|I|$.	
\end{thm:firstmilnor}

It immediately follows that linking number is an invariant of $0$-shake concordance and that for $0$-shake slice links all $\overline{\mu}$ invariants vanish.

\subsection*{Acknowledgements}
The author worked on this project under the generous mentorship of Tim Cochran and Shelly Harvey, as well as Andrew Putman; the author is indebted to them and the other faculty and graduate students, notably Aru Ray, JungHwan Park, and Miriam Kuzbery, of Rice University for many helpful conversations. He is also appreciative of Jae Choon Cha and the POSTECH mathematics department for hosting him during an extended visit that contributed significantly to the paper.

\section{Definitions and Background}

Given a link $L$, we can form the 4-manifold $W_{L}^r$ by attaching $m$ 2-handles to $B^4$; one along each component of $L$ with framing $r$. If $L$ is an $m$-component link, we have $H_2(W_{L}^r)\cong\mathbb{Z}^m$. We then say $L$ is {\it $r$-shake slice} if there exist $m$ disjoint spheres embedded in $W_L^r$ that represent the generators $(1,0,...,0),(0,1,...0),...,(0,0,...,1)$ of $H_2(W_{L}^r)$. Equivalently, we may define $r$-shake slice in terms of surfaces by introducing the idea of an $r$-shaking:

\begin{defn}
For $L$ an $m$-component link, define a {\bf $(2n_1+1, ..., 2n_m+1)-$component $r$-shaking} of $L$ to be the link formed by taking $2n_i+1$ $r$-framed parallel copies of each component $L_i$ with $n_i+1$ copies oriented in the direction of $L_i$. We will often refer to this simply as an $r$-shaking, not specifying the number of components.
\end{defn}

In Figure \ref{fig:rShakeHopf} we depict such an $r$-shaking of the Hopf link with the boxes labeled $r$ indicating that the vertical strands passing through them are given $r$ full twists.

\begin{figure}[h]
	\begin{center}
	\includegraphics[width=.25\textwidth]{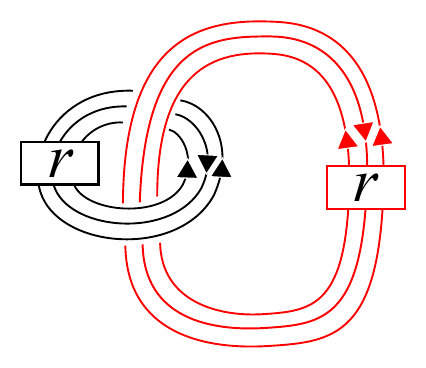}
	\caption{A $(3,3)$-component $r$-shaking of the Hopf link.}
	\label{fig:rShakeHopf}
	\end{center}
\end{figure}
\begin{defn}
We say $L\subset S^3=\partial B^4$ is {\bf $r$-shake slice} if there exist $m$ disjoint, smooth, properly embedded, compact, connected, genus zero surfaces $F_1,...,F_m$ in $D^4$ such that each $F_i$ bounds an $r$-shaking of $L_i$ and $n_{ij}$ pairs of oppositely oriented $r$-framed parallel copies of $L_j$ for $j\neq i$, such that $\sqcup_{k=1}^m F_i$ bound an $r$-shaking of $L$. If all such $n_{ij}=0$ for $j\neq i$, then we call $L$ {\bf strongly $r$-shake slice}.
\end{defn}

Given $m$-component links $L$ and $L'$, let $W_{L,L'}^r$ denote the 4-manifold formed by attaching {\it 2m} 2-handles to $S^3\times [0,1]$ along the components of $L\hookrightarrow S^3\times\{0\}$ and $L'\hookrightarrow S^3\times\{1\}$ with framing $r$. We say $L$ is {\bf $r$-shake concordant} to $L'$ if there exist $m$ embedded spheres in $W_{L,L'}^r$ representing the elements of $H_2(W_{L,L'}^r)\cong\mathbb{Z}^{2m}$ of the form $(a_1,...,a_m,b_1,...,b_m)$ where $a_i=b_i=1$ for some $1\leq i\leq m$ and $a_j=b_j=0$ for all $j\neq i$.

We can also formulate the definition of shake concordance in terms of shakings of the components of the links.
\begin{defn}
We say $m$-component links $L$ and $L'$ are {\bf $(2n_1+1,...,2n_m+1;2n_1'+1,...,2n_m'+1)$ $r$-shake concordant} if there exist disjoint, smooth, properly embedded, compact, connected, genus zero surfaces $F_1,...,F_m$ in $S^3\times[0,1]$ such that for each $1\leq i\leq m$:

\begin{itemize}
\item $F_i\cap (S^3\times\{0\})$ consists of a $2n_{ii}+1$ $r$-shaking of $L_i$ and $n_{ij}$ pairs of oppositely oriented $r$-framed parallel copies of $L_j$ for each $j\neq i$ such that $\sum_k n_{ik}=n_i$

\item $F_i\cap (S^3\times\{1\})$ consists of a $2n_{ii}'+1$ $r$-shaking of $L_i'$ and $n_{ij}'$ pairs of oppositely oriented $r$-framed parallel copies of $L_j'$ for each $j\neq i$ such that $\sum_k n_{ik}'=n_i'$.
\end{itemize}

If $n_{ij}=n_{ij}'=0$ for all $i\neq j$, then we call $L$ and $L'$ {\it strongly $r$-shake concordant}.
\end{defn}
Figure \ref{fig:link_shaking} illustrates these definitions. When we omit $r$ from the notation it is to be understood $r=0$.

\begin{figure}[ht]
\begin{center}
	\subfloat[]{\includegraphics[width=.45\textwidth]{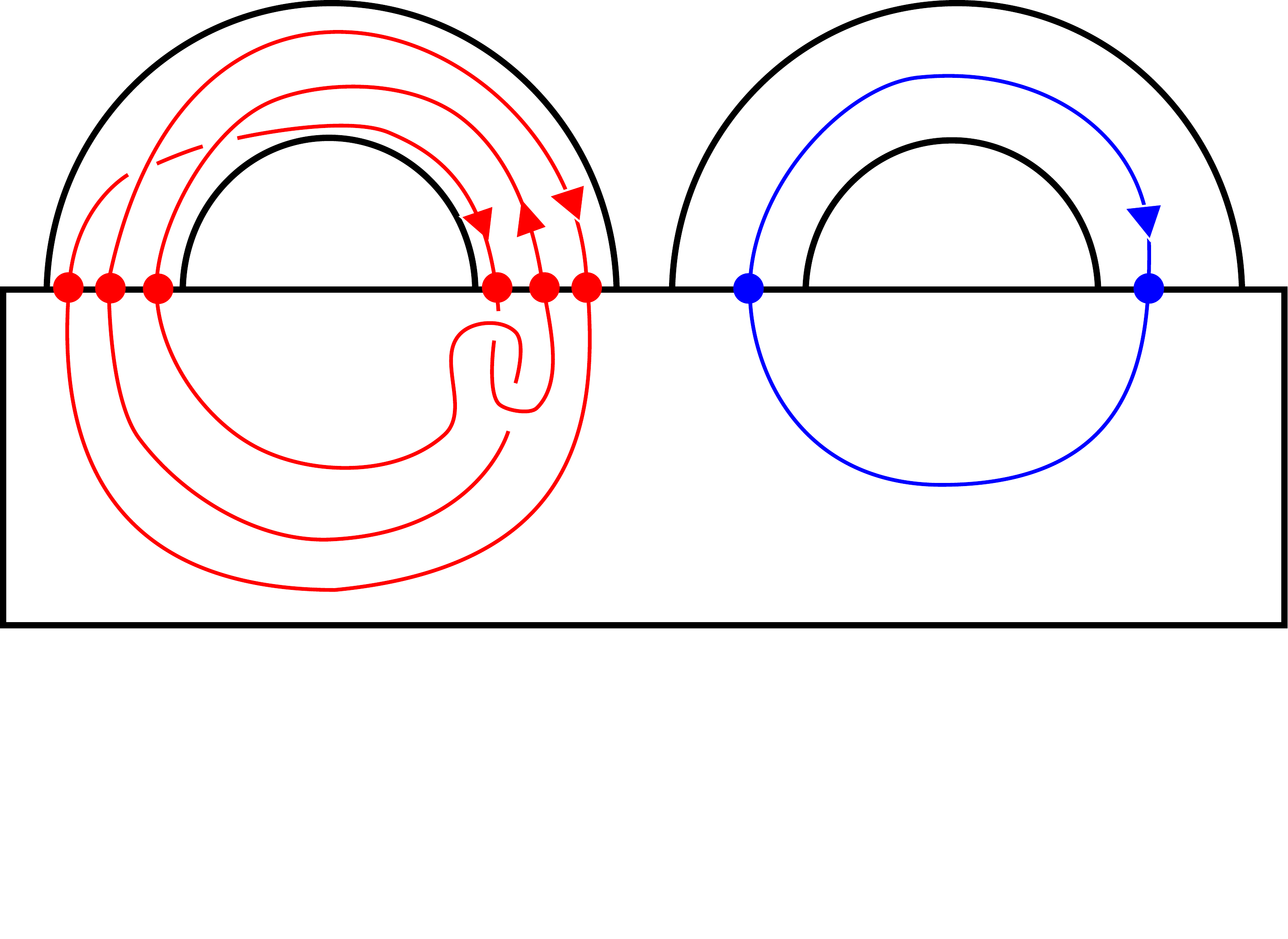}}
		\hspace{0.2in}
	\subfloat[]{\includegraphics[width=.45\textwidth]{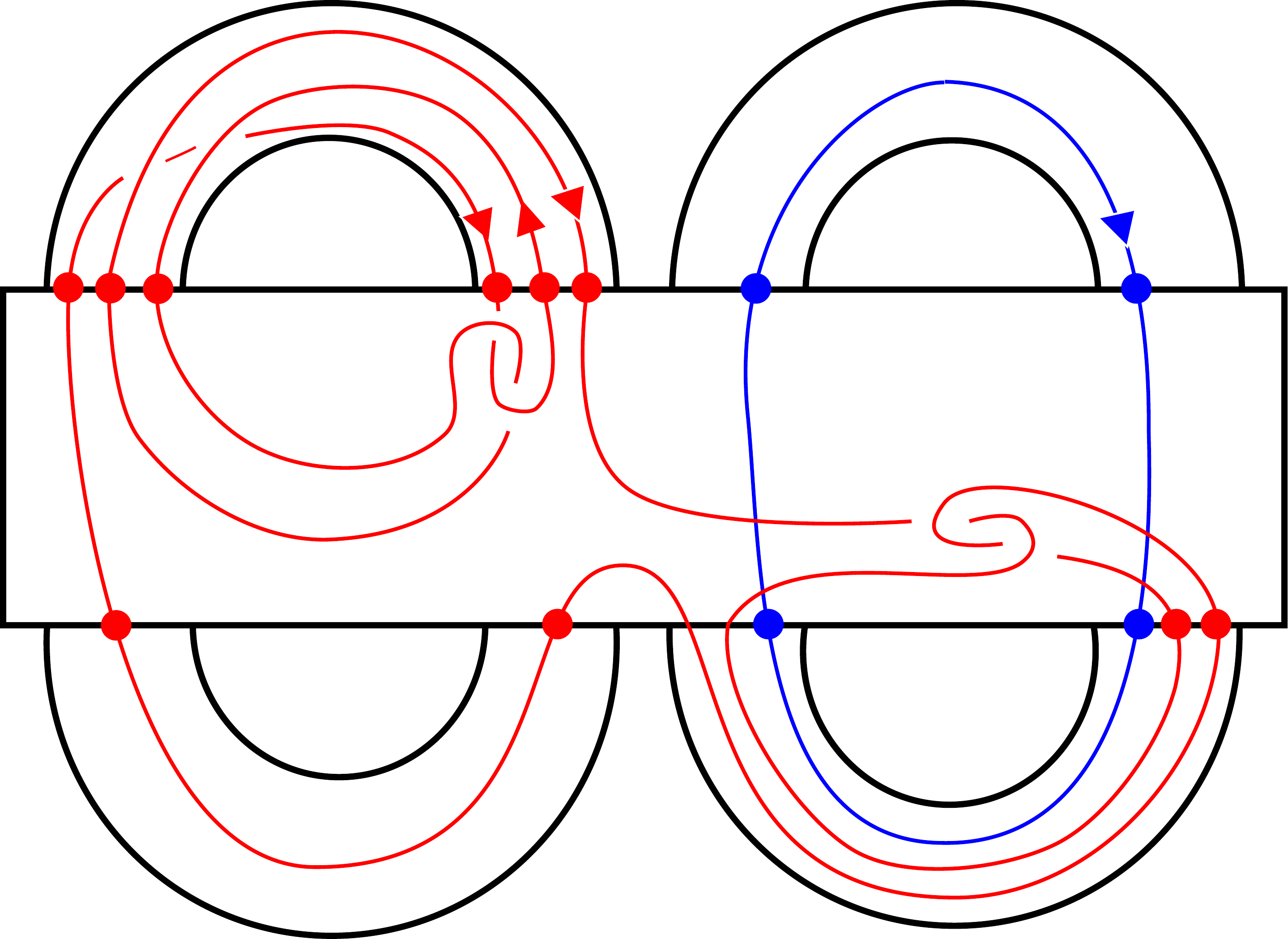}}\\
	\subfloat[]{\includegraphics[width=.45\textwidth]{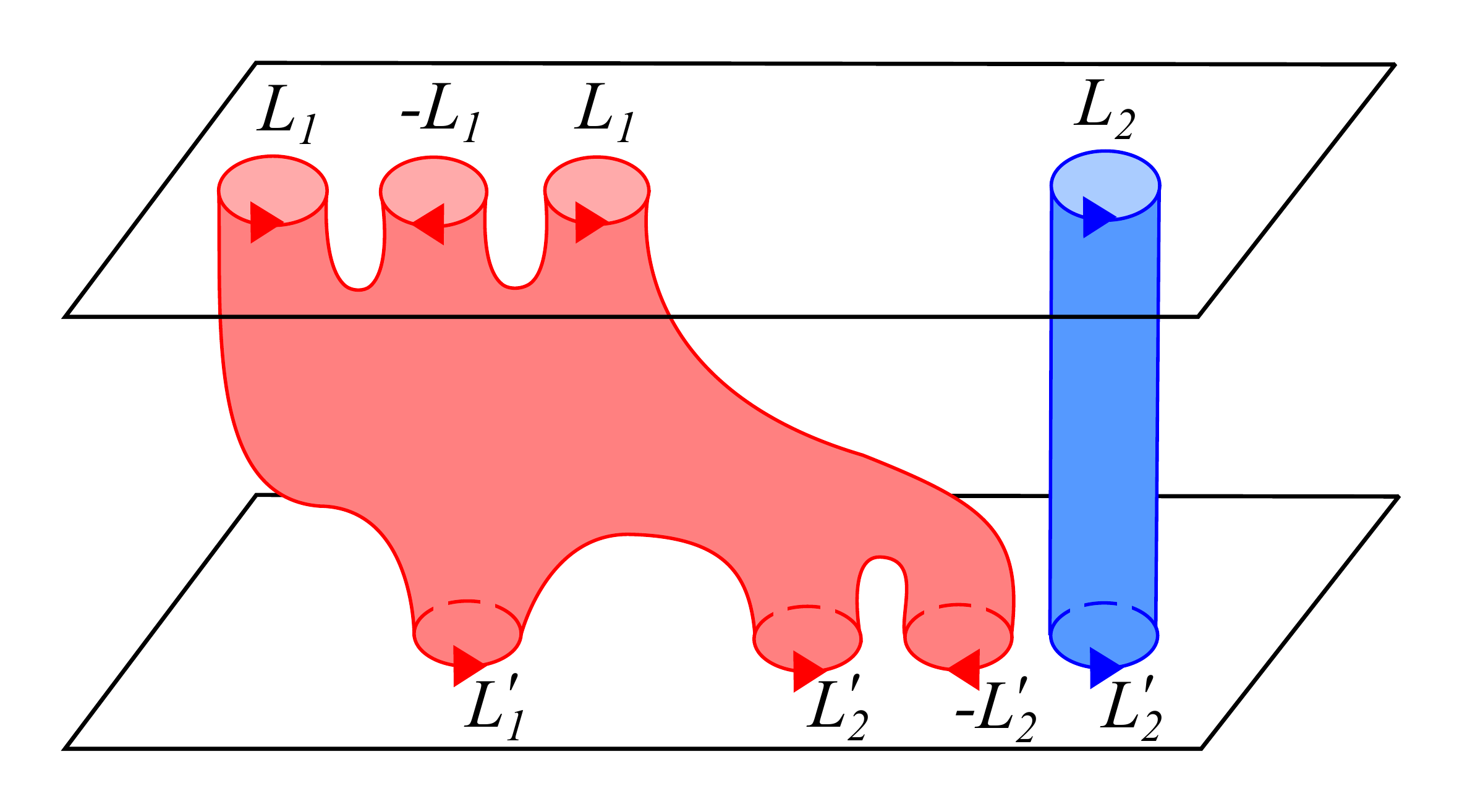}}
		\hspace{0.2in}
	\subfloat[]{\includegraphics[width=.45\textwidth]{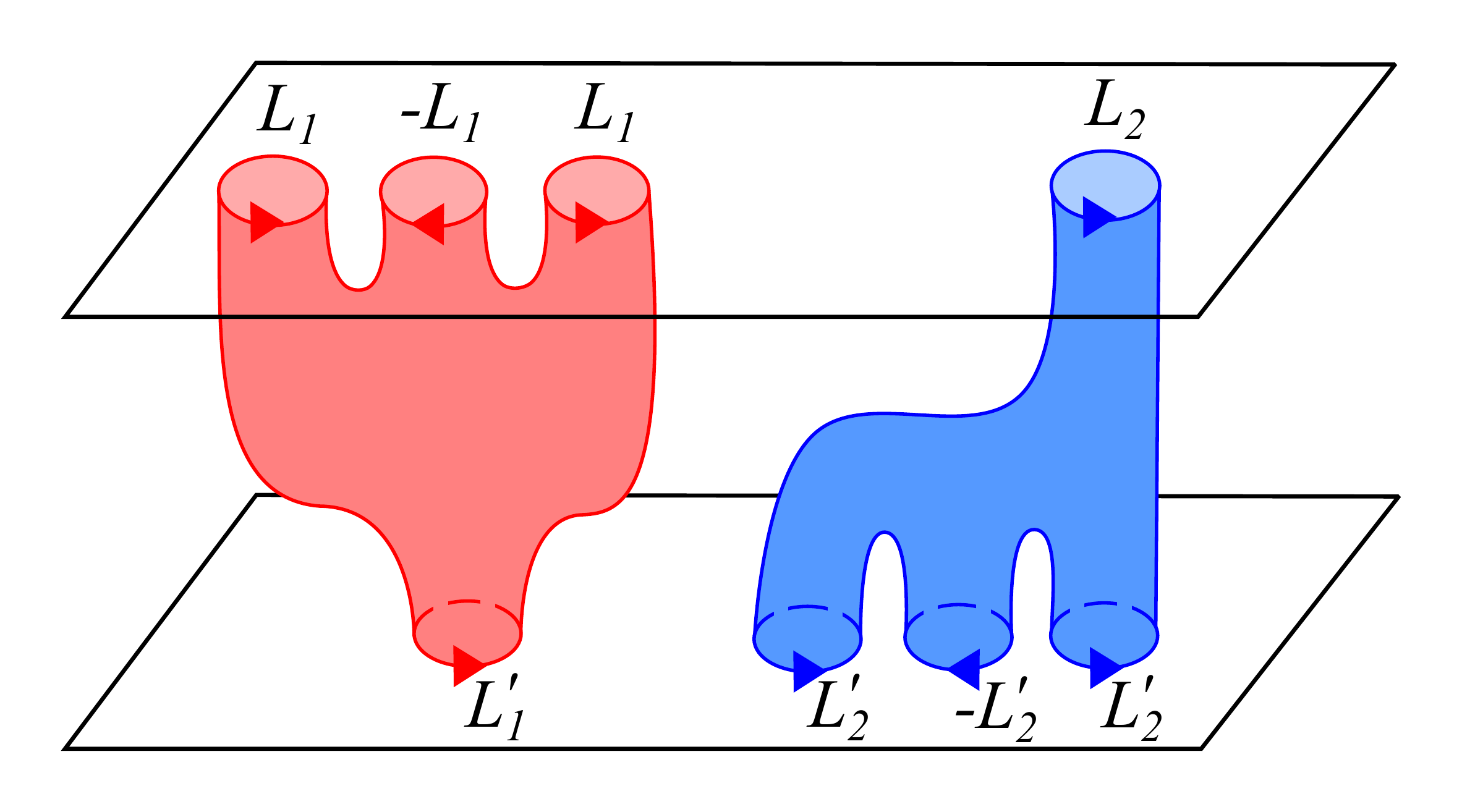}}
	\caption{Depictions of (A) homological view of a shake slice link, (B) homological view of shake concordant links, (C) $r$-shaking view of shake concordant links, and (D) $r$-shaking view of strongly shake concordant links.}
	\label{fig:link_shaking}
\end{center}
\end{figure}

For knots, strong shake concordance is equivalent to shake concordance. Moreover, if links $L=L_1\sqcup...\sqcup L_m$ and $L'=L_1'\sqcup...\sqcup L_m'$ are $(n_1,...,n_m;n_1',...,n_m')$ strongly $r$-shake concordant, then components $L_i$ and $L_i'$ are $(n_i,n_i')$ $r$-shake concordant as knots. Also, note that an $m$-component link is (strongly) $r$-shake slice if and only if it is $(n_1,...,n_m;1,...,1)$ (strongly) $r$-shake concordant to the trivial link.

\section{Shake Concordant Links That Are Not Concordant}
In this section we show the notions of concordance, $0$-shake concordance, and strong $0$-shake concordance differ from each other.

Let $h(K)$ denote the 2-component link consisting of first component $K$ and second component a meridian of $K$ as in Figure \ref{fig:ShakeNotStrongEx}.

\begin{figure}[ht]
	\begin{center}
	\includegraphics[width=.2\textwidth]{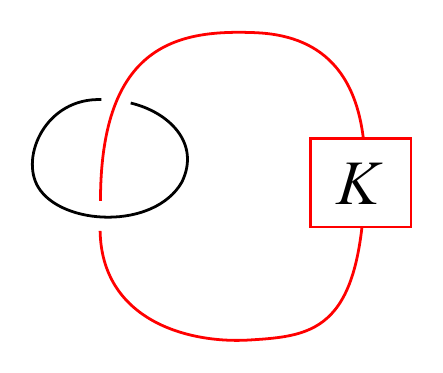}
	\caption{The 2-component link $h(K)$ consisting of $K$ and its meridian.}
	\label{fig:ShakeNotStrongEx}
	\end{center}
\end{figure}

\begin{prop}
	The links $h(K)$ and $h(J)$ are shake concordant for any knots $K$ and $J$.
	\label{prop:crossings}
\end{prop}
\begin{proof}
	Figure \ref{fig:crossing_change} depicts how a shake concordance accomplishes a crossing change in a link of the form $h(K)$. Position the meridian component so that it is next to the crossing in $K$ that we want to change as in Figure \ref{fig:crossing_change} (A). Then, consider a $(1,3)$ shaking of $h(K)$ as in Figure \ref{fig:crossing_change} (B). Take a band sum of one of the meridian components that has the appropriate orientation with $K$ at the crossing to change it from an overcrossing to undercrossing (or vice versa). Also take a band sum of a meridian component that has opposite orientation and band it with $K$ so as not to change $K$ as in as in Figure \ref{fig:crossing_change} (Ci).  Depending on orientations, we may need a half twist in the band as in Figure \ref{fig:crossing_change} (Cii). Attaching these bands accomplishes the desired crossing change as in Figure \ref{fig:crossing_change} (D).
	
	Notice this technique can be extended to accomplish any number of crossing changes via a shake concordance. In particular, there is a $(1,2n+1;1,1)$ shake concordance between $h(K)$ and $h(J)$ if $K$ and $J$ differ by $n$ crossing changes. \end{proof}

\begin{figure}[ht]
	\captionsetup[subfloat]{labelformat=empty}
	\subfloat[(A)]{\includegraphics[width=.22\textwidth]{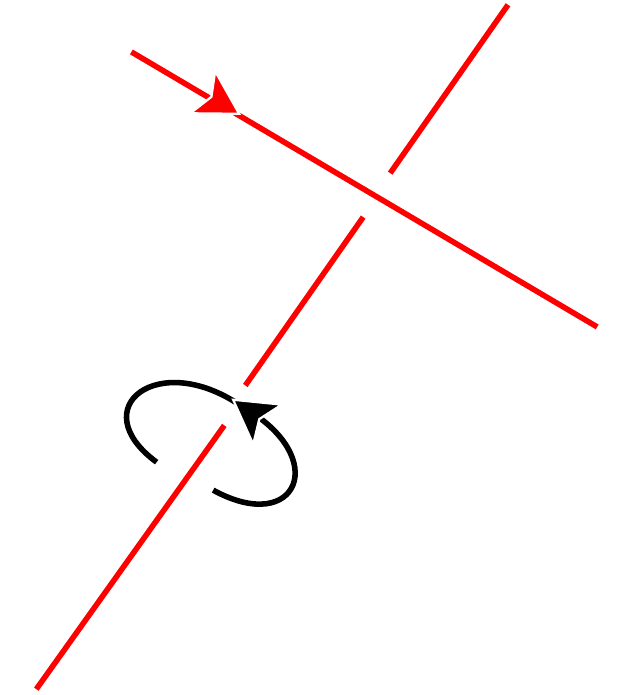}}
	\subfloat[(B)]{\includegraphics[width=.22\textwidth]{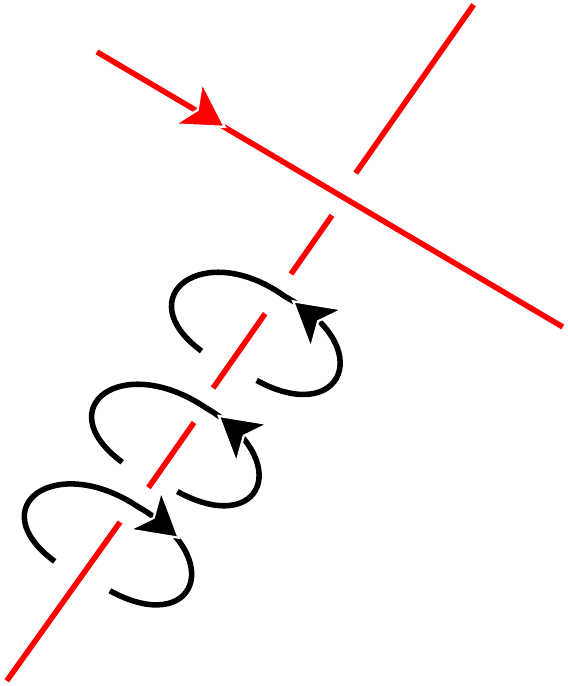}}
	\subfloat[(Ci)]{\includegraphics[width=.21\textwidth]{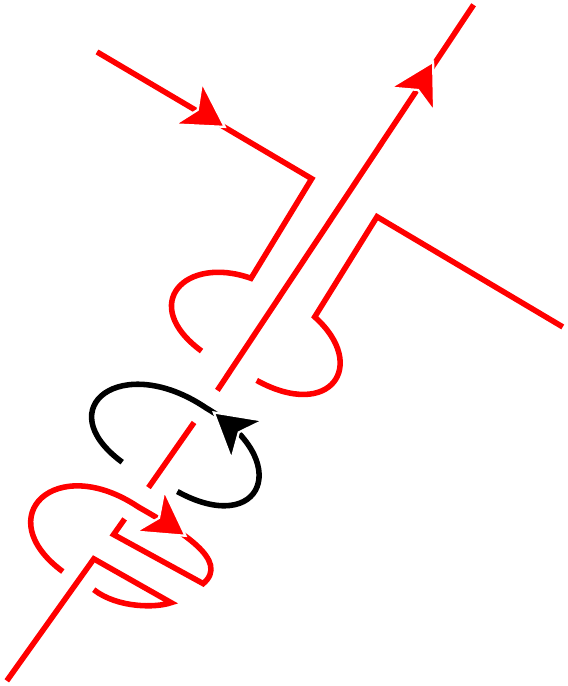}}
	\subfloat[(Cii)]{\includegraphics[width=.2\textwidth]{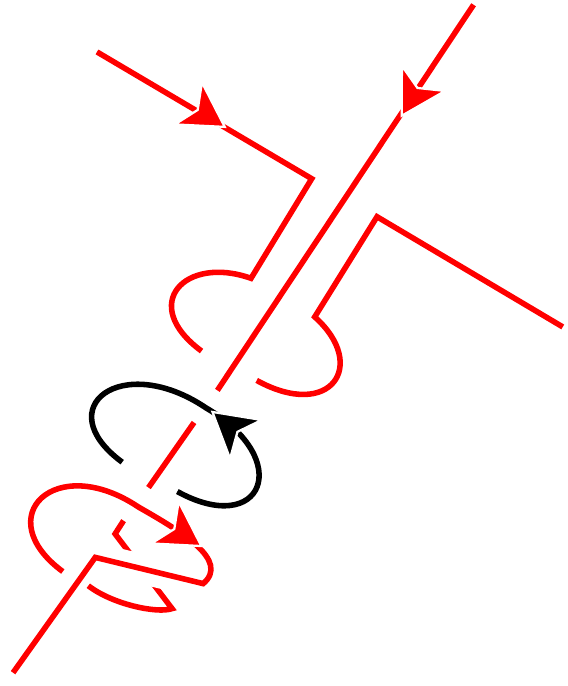}}
	\subfloat[(D)]{\includegraphics[width=.2\textwidth]{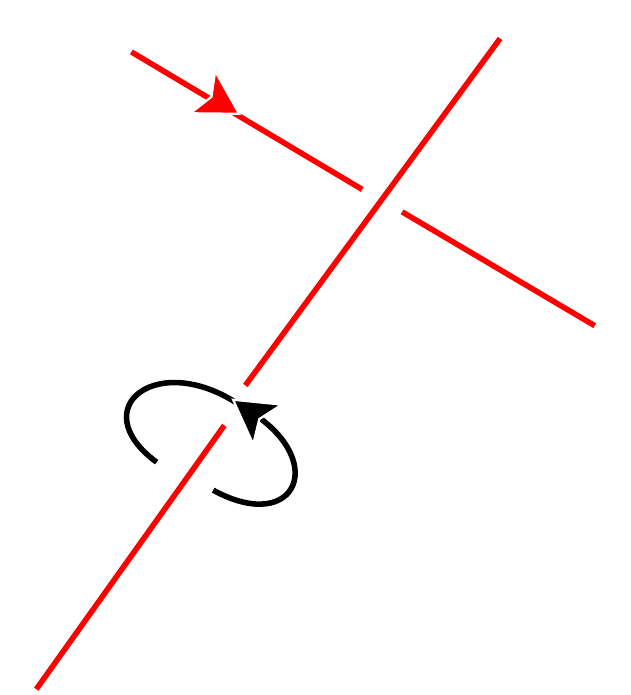}}
	\caption{Steps to accomplish crossing change via shake concordance.}
	\label{fig:crossing_change}
\end{figure}

It follows from Proposition \ref{prop:crossings} that corresponding sublinks of shake concordant links are not necessarily shake concordant! Moreover, since certain knot signatures are invariants of shake concordance \cite{akbulut77}, if we consider $K$ and $J$ of differing signatures, then $h(K)$ and $h(J)$ are not strongly shake concordant. It follows:

\begin{cor}
\label{cor:shakenotstrong}
There exists an infinite family of 2-component links that are pairwise shake concordant, but not pairwise strongly shake concordant.
\end{cor}

\begin{proof}
Consider, for instance, the family $\{h(T_k)\}_{k=1,2,3,...}$ where $T_k$ is the connected sum of $k$ trefoil knots. As each $T_k$ has signature $\pm2k$, no two links in the family are pairwise strongly shake concordant.
\end{proof}

The proof of Proposition \ref{prop:crossings} suggests the following result.

\begin{prop}
	If two $m$-component links are link homotopic, then they are sublinks of shake concordant $2m$-component links.
\end{prop}

\begin{proof}
Suppose $m$-component links $L$ and $L'$ are link homotopic, then $L'$ can be obtained from $L$ by ambient isotopy and crossing changes between arcs of the same component of $L$. However, as in the proof of Proposition \ref{prop:crossings}, these crossing changes can also be obtained via shake concordance of the components of $L$ with added meridian components. Therefore, if we let $J$ (resp. $J'$) denote the $2m$-component link consisting of the $m$ components of $L$ (resp. $L'$) and $m$ meridian components, one for each component of $L$, then we have $J$ is shake concordant to $J'$. See Figure \ref{fig:HomotopicLinksShake}.
\end{proof}

\begin{figure}[!h]
	\begin{center}
	\subfloat[Homotopic links.]{\includegraphics[width=.5\textwidth]{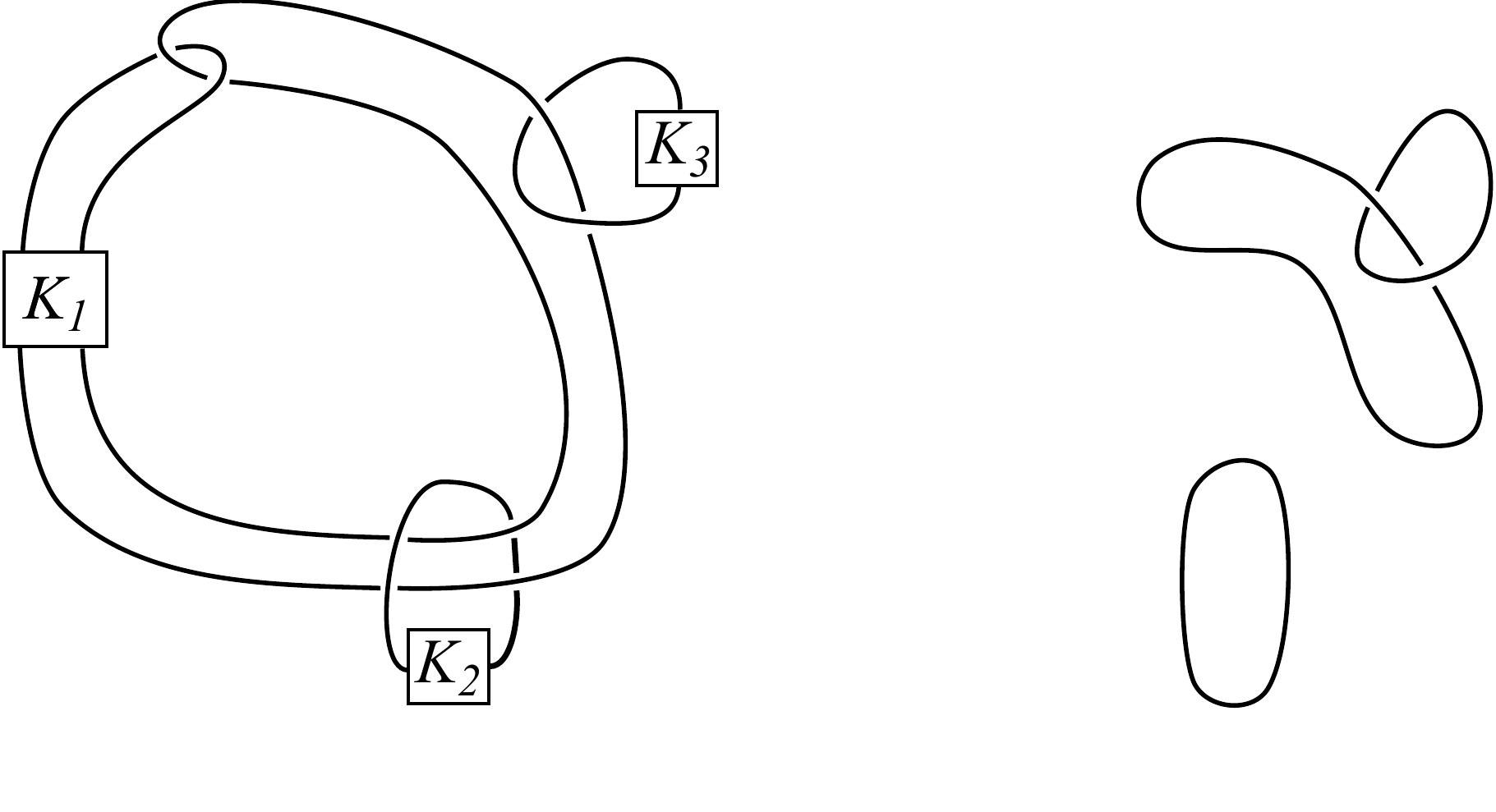}}\\
	\subfloat[Shake concordant links.]{\includegraphics[width=.5\textwidth]{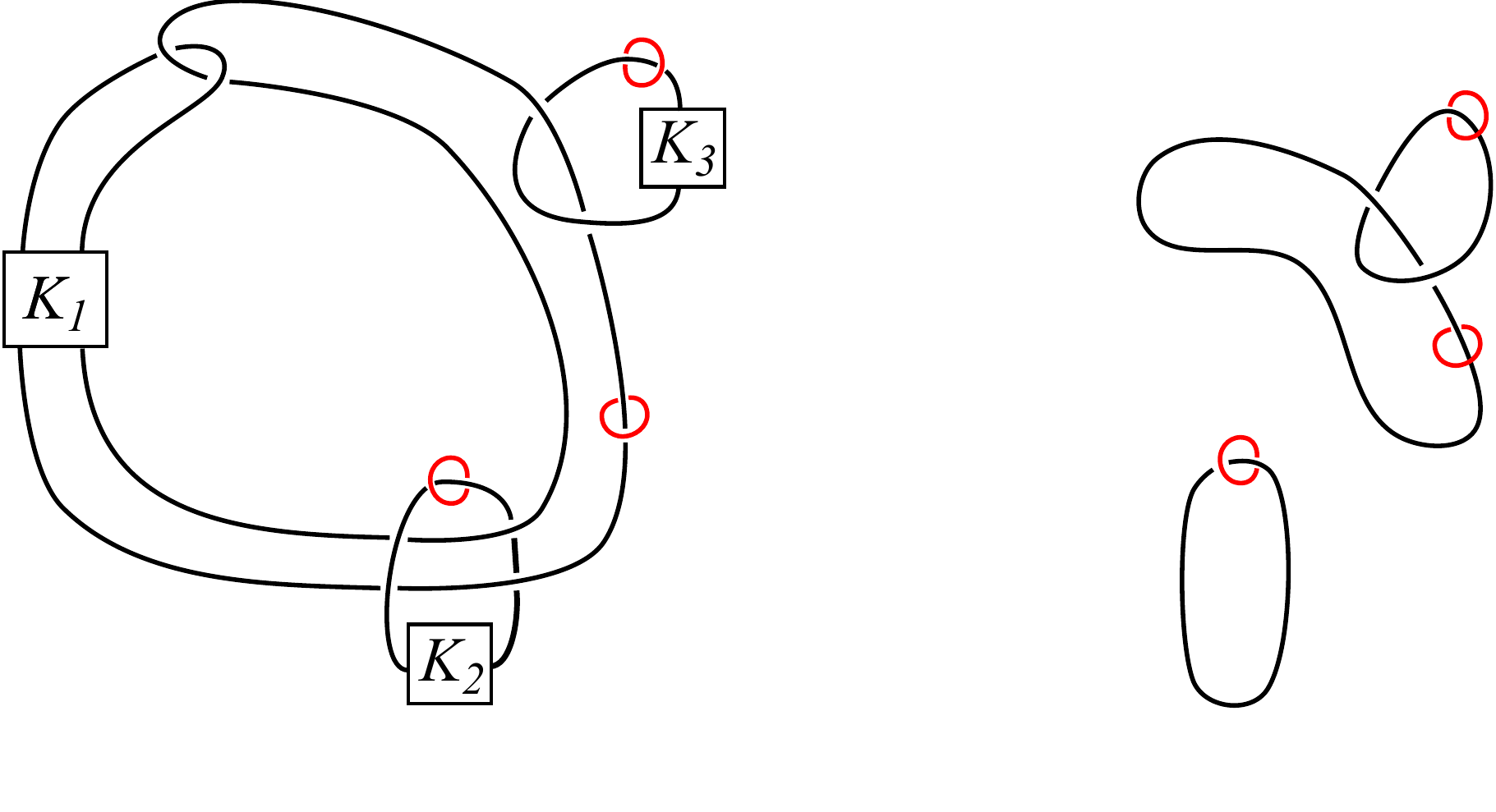}}
	\caption[Homotopic and shake concordant links.]{}
	\label{fig:HomotopicLinksShake}
	\end{center}
\end{figure}

We now show that there is an infinite family of 2-component links with unknotted components that are pairwise strongly shake concordant but not shake concordant.

Given a knot $K$, let $L(K)$ denote the 2-component link of Figure \ref{fig:ChaLink}. Note that each component of $L(K)$ is unknotted and $L(U)$ is the Hopf link. It is known that $\tau(K)$, the Ozsv\'ath-Szab\'o tau invariant for $K$ defined in \cite{os03}, can distinguish $L(K)$ from the Hopf link (Theorem 4.1 in \cite{ckrs}):

\begin{figure}[ht]
	\begin{center}
	\includegraphics[width=.2\textwidth]{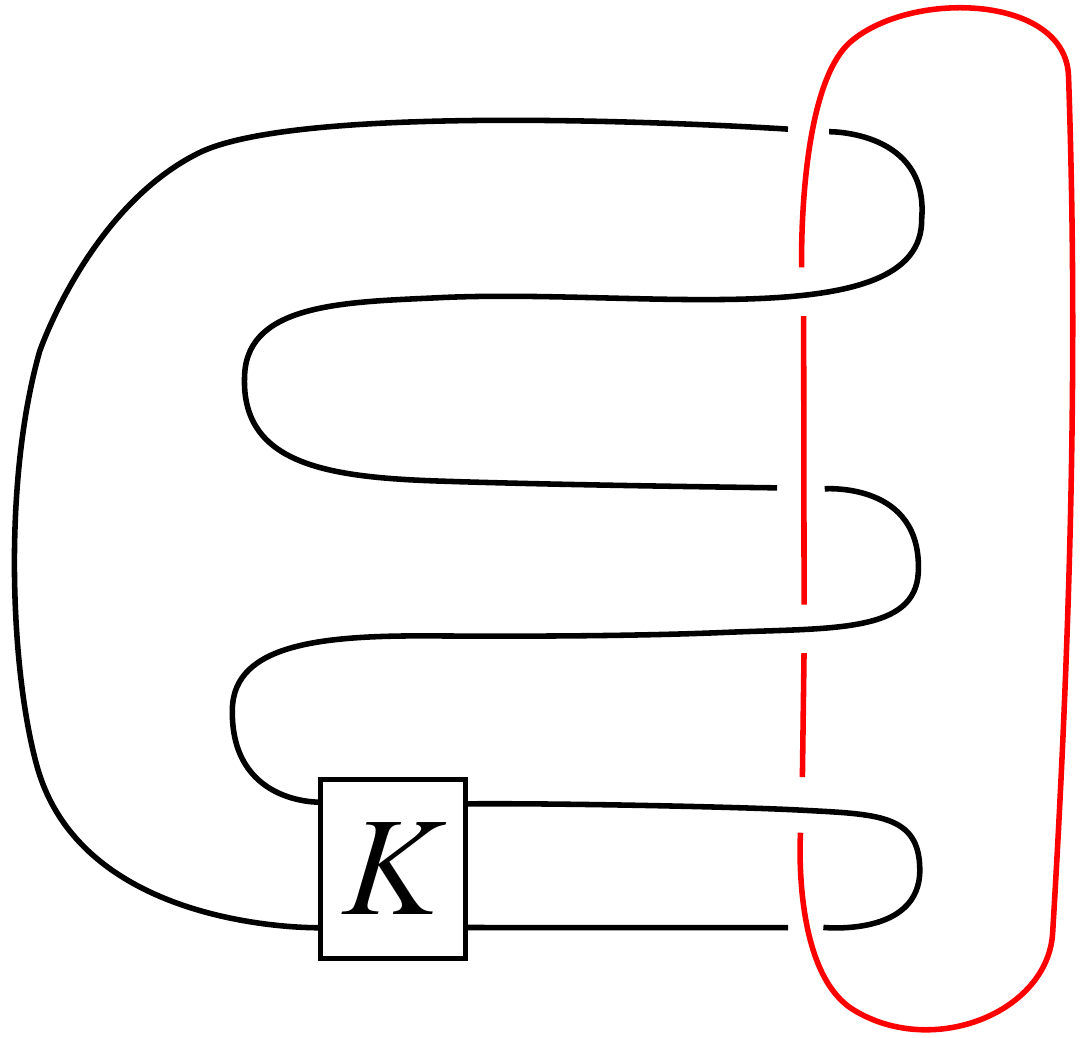}
	\caption{The link $L(K)$.}
	\label{fig:ChaLink}
	\end{center}
\end{figure}

\begin{thm}[Cha, Kim, Ruberman, Strle 2010]
\label{thm:cha}
		The 2-component link $L(K)$ is not concordant to the Hopf link for $K$ when $\tau(K)>0$. Moreover, there is a infinite family of knots $K_n$ such that $L(K_n)$ are distinct up to smooth concordance.
\end{thm}

It then follows:

\begin{prop}
	\label{prop:strongbutnotconc}
	There exists an infinite family of 2-components links with trivial components that are all strongly shake concordant to the Hopf link, but none of which are concordant to the Hopf link.
\end{prop}
\begin{proof}
We argue that $L(K)$ is strongly $(1,1;3,1)$ shake concordant to the Hopf link for any knot $K$. First take a $(3,1)$ shaking of the Hopf link as in Figure \ref{fig:cha_link_hopf} (A) and (B). Then note Figure \ref{fig:cha_link_hopf} (B) is ambient isotopic to Figure \ref{fig:cha_link_hopf} (C) wherein we may band sum the parallel copies to obtain Figure \ref{fig:cha_link_hopf} (D). This gives the desired strong shake concordance; however, we have by Theorem \ref{thm:cha} that $L(K)$ is not concordant to the Hopf link whenever $\tau(K)>0$.
\end{proof}

\begin{figure}[ht]
	\begin{center}
	\subfloat[]{\includegraphics[width=.1\textwidth]{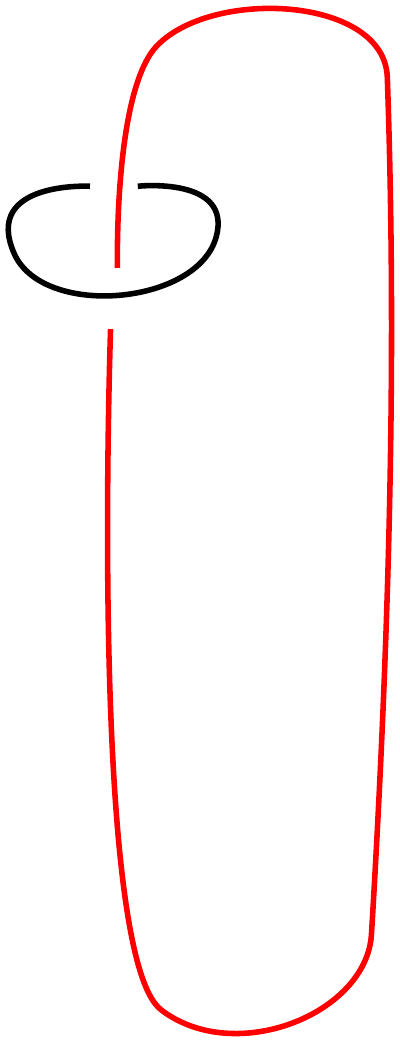}}
	\hspace{.05\textwidth}
	\subfloat[]{\includegraphics[width=.1\textwidth]{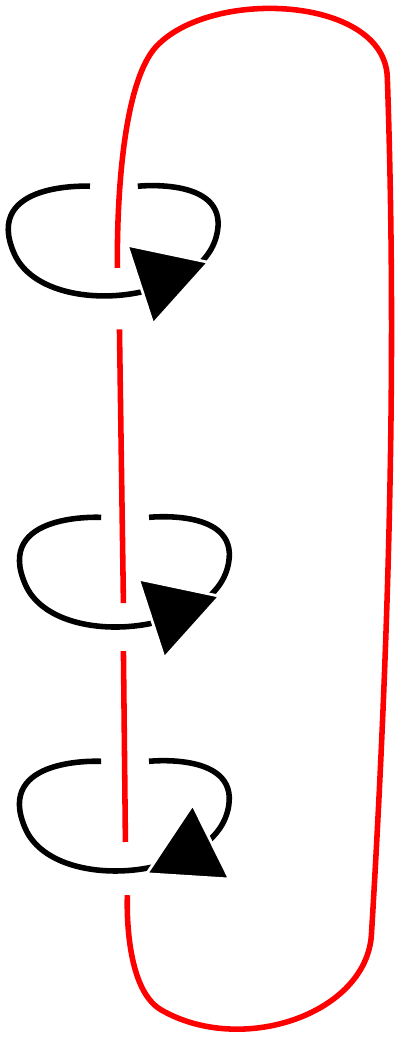}}
	\hspace{.05\textwidth}
	\subfloat[]{\includegraphics[width=.28\textwidth]{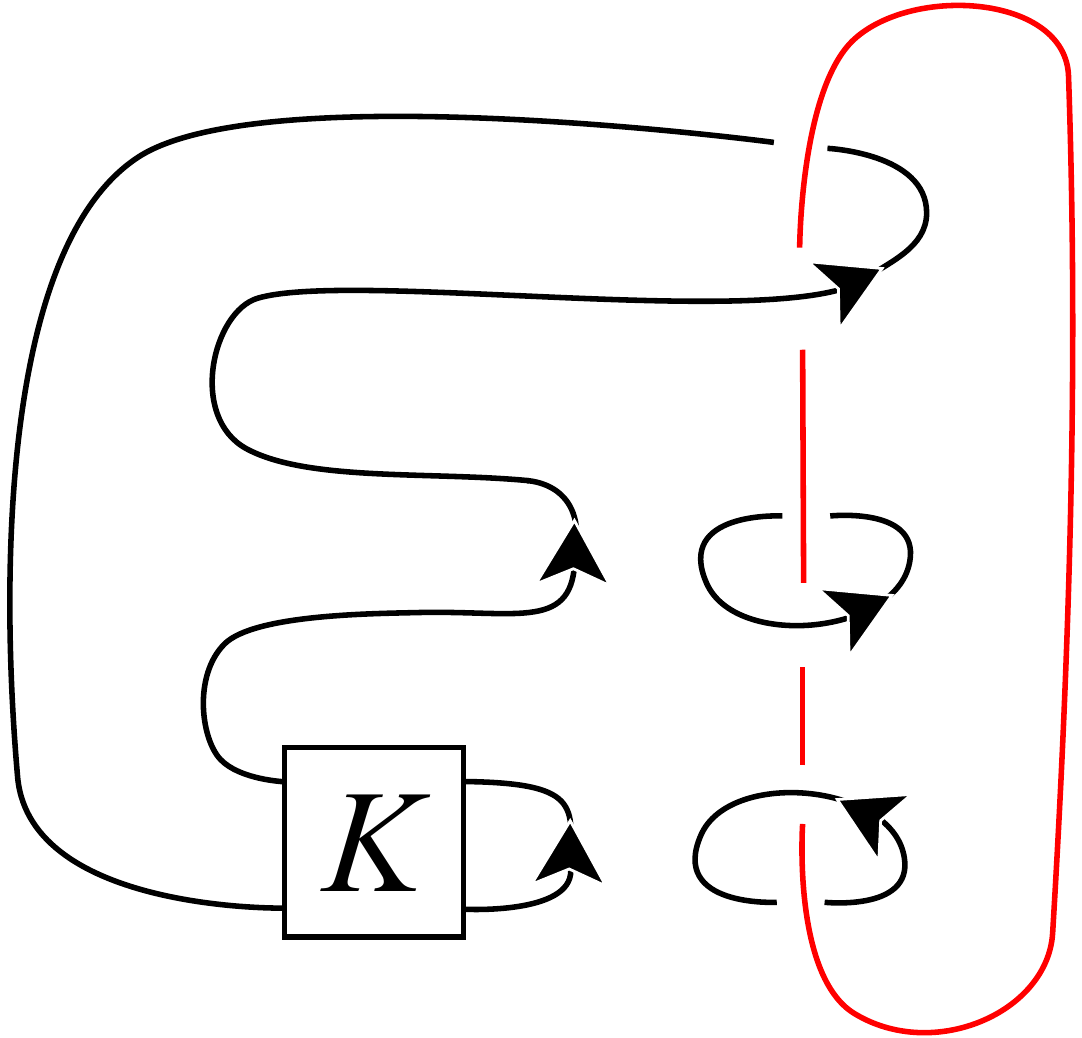}}
	\hspace{.05\textwidth}
	\subfloat[]{\includegraphics[width=.28\textwidth]{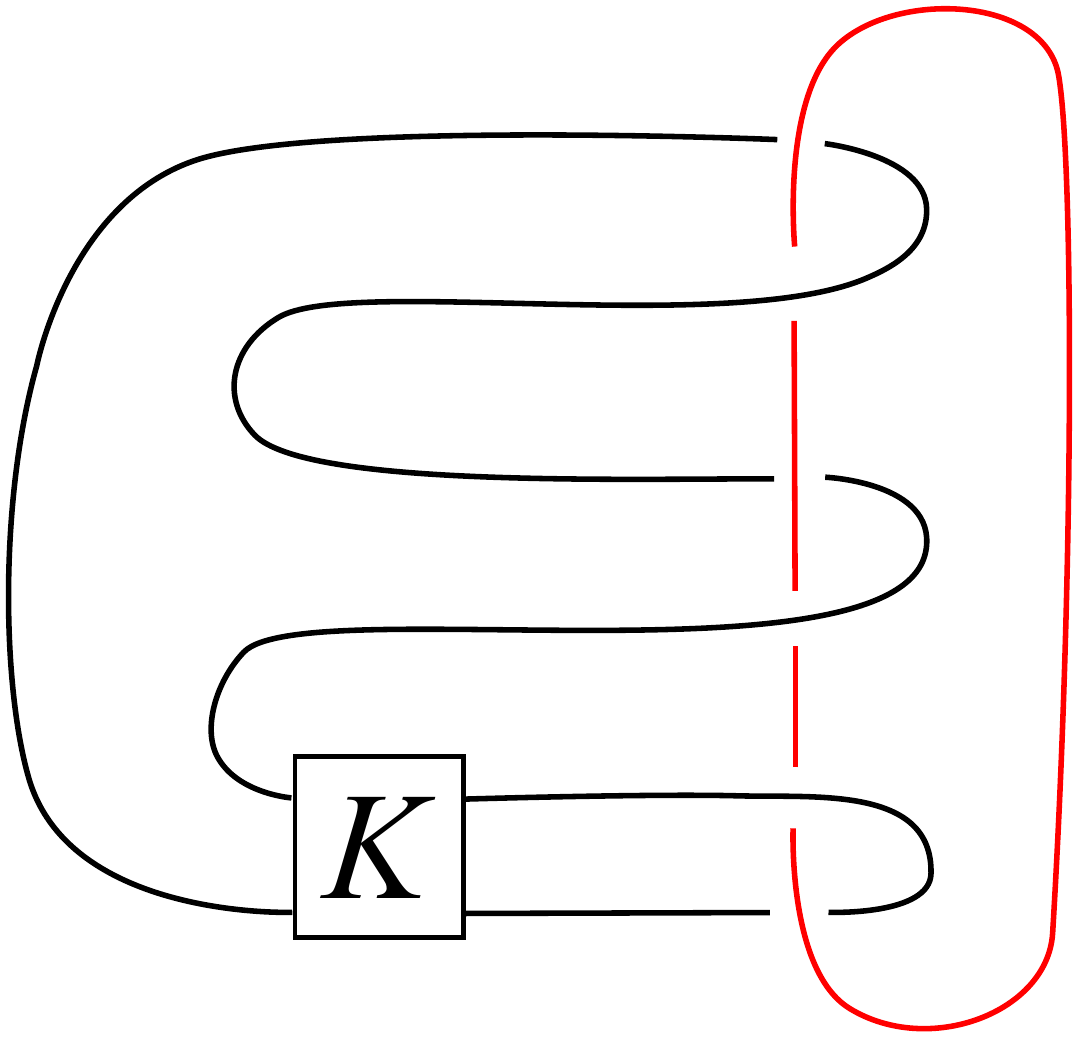}}
	\caption{Steps to obtain strong shake concordance between Hopf link and $L(K)$.}
	\label{fig:cha_link_hopf}
	\end{center}
\end{figure}

\section{Classification of Links Up of Shake Concordance}

In this section we characterize shake concordance of links in terms of concordance and string link infections, generalizing the known classification for shake concordance of knots in terms of satellite operators:

\begin{thm}[Cochran, Ray 2015]
	The knots $K$ and $J$ are shake concordant if and only if there exist winding number one patterns $P$ and $Q$, with $P$ and $Q$ ribbon knots, such that $P(K)$ is concordant to $Q(J)$.
\end{thm}

Let $I(L,J,\mathbb{E}_\varphi)$ denote the multi-infection of a $m$-component link $L$ by $k$ component string link $J$ along the image of the embedding of a multidisk $\varphi: \mathbb{E}\rightarrow S^3$ where disk $\mathbb{E}$ has $k$ subdisks $D_1,...,D_k$. Denote the image of the embedding by $\mathbb{E}_\varphi$, then the multi-infection takes place in the thickened disk $\mathbb{E}_\varphi\times[0,1]$. Effectively, this ties $L$ into $J$ along $\mathbb{E}_\varphi\times[0,1]$ resulting in the infected link $I(L,J,\mathbb{E}_\varphi)\subset S^3$. See \cite[Section 2.2]{cft} for a formal definition.

We say $\mathbb{E}_\varphi$ {\it respects} $L$ if $k=m$ and each link component $L_i$ intersects the subdisk $D_j$ algebraically once when $i=j$ and algebraically zero times when $i\neq j$. If, moreover, the count of intersection points is geometrically zero wherever $i\neq j$, then we say $\mathbb{E}_\varphi$ {\it strongly respects} $L$.

For the characterization theorem, we will need the following lemma.
\begin{lem}
Given slice $m$-component link $L$, string link $J$, and an embedded multidisk $\mathbb{E}_\varphi$ that respects $L$. Let $n_i$ denote the number of times that $L$ geometrically intersects the subdisk $D_i$ of $\mathbb{E_\varphi}$, for $i=1,...,m.$ Then we have that $I(L,J,\mathbb{E}_\varphi)$ is $(1,...,1;n_1,...,n_m)$ shake concordant to $\widehat{J}$, the closure of $J$.
\label{lemma:infectedshake}
\end{lem}

\begin{proof} Note that $L$ can be obtained by band summing a copy of $L$ with a $(n_1,...,n_m)$ shaking of the $m$-component unlink, which we will denote $T_m$, as in Figure \ref{fig:lemma}.

\begin{figure}[ht]
	\begin{center}
	\includegraphics[width=.4\textwidth]{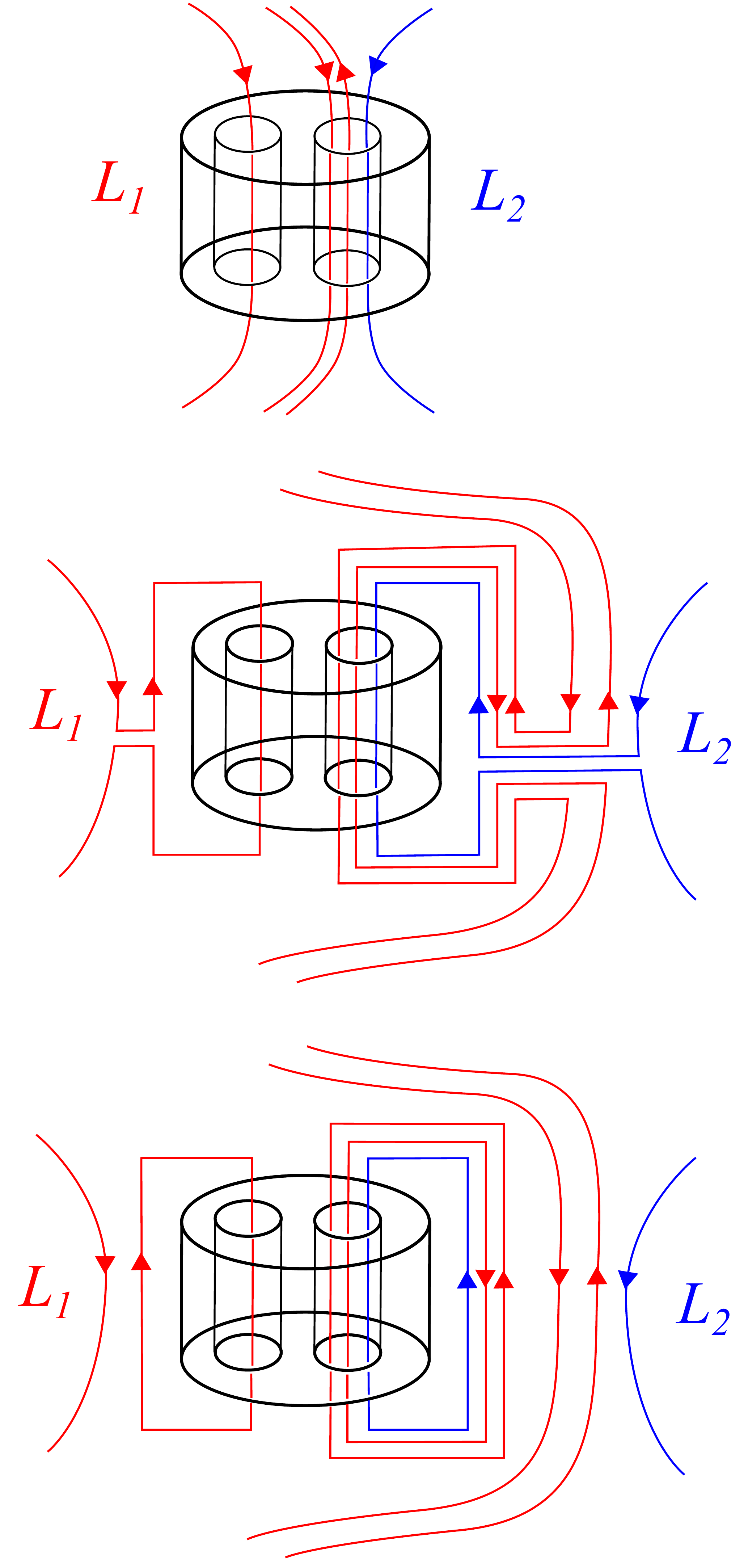}
	\caption{Effect of bands on fusing components.}
	\label{fig:lemma}
	\end{center}
\end{figure}

Hence there exists a smooth, compact, connected, genus zero surface $S\subset S^3\times[0,1]$ that cobounds $L\subset S^3\times\{0\}$ and the disjoint union $L\sqcup T_m\subset S^3\times\{1\}$.

We may construct $S$ such that it lies entirely in the complement of \[\left((\mathbb{E}_\varphi/\sqcup_i\varphi(D_i))\times[0,1]\right)\times[0,1]\subset S^3\times[0,1].\] Thus, we may multi-infect the link at each $t\in [0,1]$ along $\left(\mathbb{E}_\varphi\times[0,1]\right)\times\{t\}$ by the string link $J$, thus modifying the surface $S$ so that it cobounds $I(L,J,\mathbb{E}_\varphi)\subset S^3\times\{0\}$ and a disjoint union of $L$ and a $(n_1,...,n_m)$ shaking of $\widehat{J}$ in $S^3\times\{1\}$. As $L$ is slice, we can cap the components of $L$ off to obtain a $(1,...,1;n_1,...,n_m)$ shake concordance between $I(I,J,\mathbb{E}_\varphi)$ and $\widehat{J}$.
\end{proof}

\begin{cor}
Consider $m$-component slice links $L$ and $L'$, $m$-component string links $J$ and $J'$, and embeddings of multidisks $\mathbb{E}_\varphi$ respecting $L$ and $\mathbb{E'}_{\varphi'}$ respecting $L'$. If $I(L,J,\mathbb{E}_\varphi)$ is concordant to $I(L',J',\mathbb{E'}_{\varphi'})$, then $\widehat{J}$ is shake concordant to $\widehat{J'}$.
\label{cor:concthenshake}
\end{cor}
\begin{proof}
If $I(L,J,\mathbb{E}_\varphi)$ is concordant to $I(L',J',\mathbb{E'}_{\varphi'})$ then by Lemma \ref{lemma:infectedshake} $I(L,J,\mathbb{E}_\varphi)$ is $(1,...,1;n_1,...,n_m)$ shake concordant to $\widehat{J}$ and $I(L',J',\mathbb{E'}_{\varphi'})$ is $(1,...,1;n'_1,...,n'_m)$ shake concordant to $\widehat{J'}$. Hence, $\widehat{J}$ is $(n_1,...,n_m;n'_1,...,n'_m)$ shake concordant to $\widehat{J'}$.
\end{proof}

The following lemma serves as a converse to Lemma \ref{lemma:infectedshake}.
\begin{lem}
If $L$ is $(1,...,1; n_1,...,n_m)$ shake concordant to $L'$, then $L$ is concordant to $I(S,J',\mathbb{E}_\varphi)$ for some slice link $S$, string link $J'$ with closure $\widehat{J'}=L'$, and embedded multidisk $\mathbb{E}_\varphi$ that respects $S$.
\label{lemma:shakeimplies}
\end{lem}

\begin{proof}
Since $L$ is $(1,...,1; n_1,...,n_m)$ shake concordant to $L'$, there exist $m$ disjoint genus zero surfaces in $S^3\times [0,1]$, which we will denote $F_1,...,F_m$, with boundary $L\subset S^3\times \{0\}$ and $sh(L')\subset S^3\times \{1\}$ where $sh(L')$ denotes a $(n_1,...,n_m)$ shaking of $L'$.

We can move each $F_i$ via ambient isotopy such that the projection map $S^3\times [0,1]\rightarrow [0,1]$ is a Morse function when restricted to each $F_i$ and such that all local maxima occur at level $\{\frac{4}{5}\}$, split saddles at level $\{\frac{3}{5}\}$, join saddles at level $\{\frac{2}{5}\}$, and local minima at level $\{\frac{1}{5}\}$. Hence, the level $\{\frac{1}{2}\}$ of each surface $F_i$ is a connected component $M_i$ of some $m$-component link $M$. Notice, $L$ is concordant to $M$. See Figure \ref{fig:morse_map}.

\begin{figure}[ht]
	\begin{center}
	\includegraphics[width=.5\textwidth]{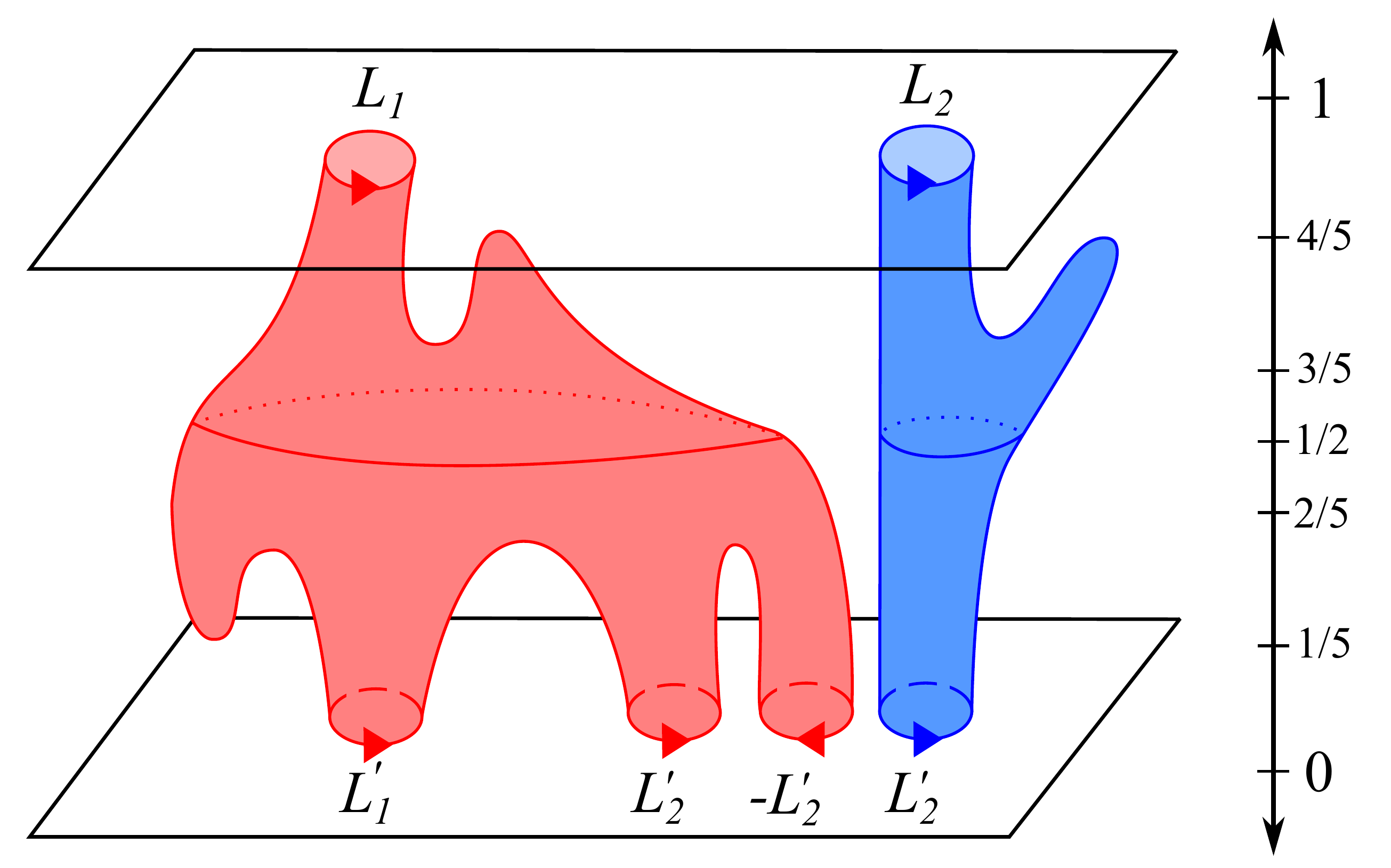}
	\caption{Morse function on the shake concordance.}
	\label{fig:morse_map}
	\end{center}
\end{figure}

Moreover, $M$ is a fusion of $sh(L')$ and a trivial link $T$, where each component of $T$ corresponds to a local minima of some $F_i$. That is, $M$ is obtained by fusing distinct components of $sh(L')$ and $T$ via band sum to obtain an $m$-component link.

Let $T_N$ denote a $(n_1,...,n_m)$ shaking of a $m$-component trivial link $T_m$, which is itself a trivial link of $N=\Sigma_i n_i$ components. Consider $\mathbb{E}_\varphi$ such that the $n_i$ components of $T_N$ that are parallel copies of the $i^{th}$ component of $T_m$ intersect subdisk $D_j$ of $\mathbb{E}_i$ algebraically $\delta_{ij}$ times. Then $\mathbb{E}_\varphi$ strongly respects $T_N$ and $sh(L')=I_r(T_N,J',\mathbb{E}_\varphi)$ for string link $J'$ with closure $\widehat{J'}=L'$.

If via ambient isotopy we can make the fusion bands avoid $\mathbb{E}_\varphi\times[0,1]$, then we may express $M=I(T',J',\mathbb{E}_\varphi)$ where $T'$ is a link obtained by fusing $T_N$ and $T$; see Figure \ref{fig:morse2}. As $T'$ is then slice and $L$ is concordant to $M=I(T',J',\mathbb{E}_\varphi)$, we are finished.

\begin{figure}[ht]
	\begin{center}
	\includegraphics[width=.4\textwidth]{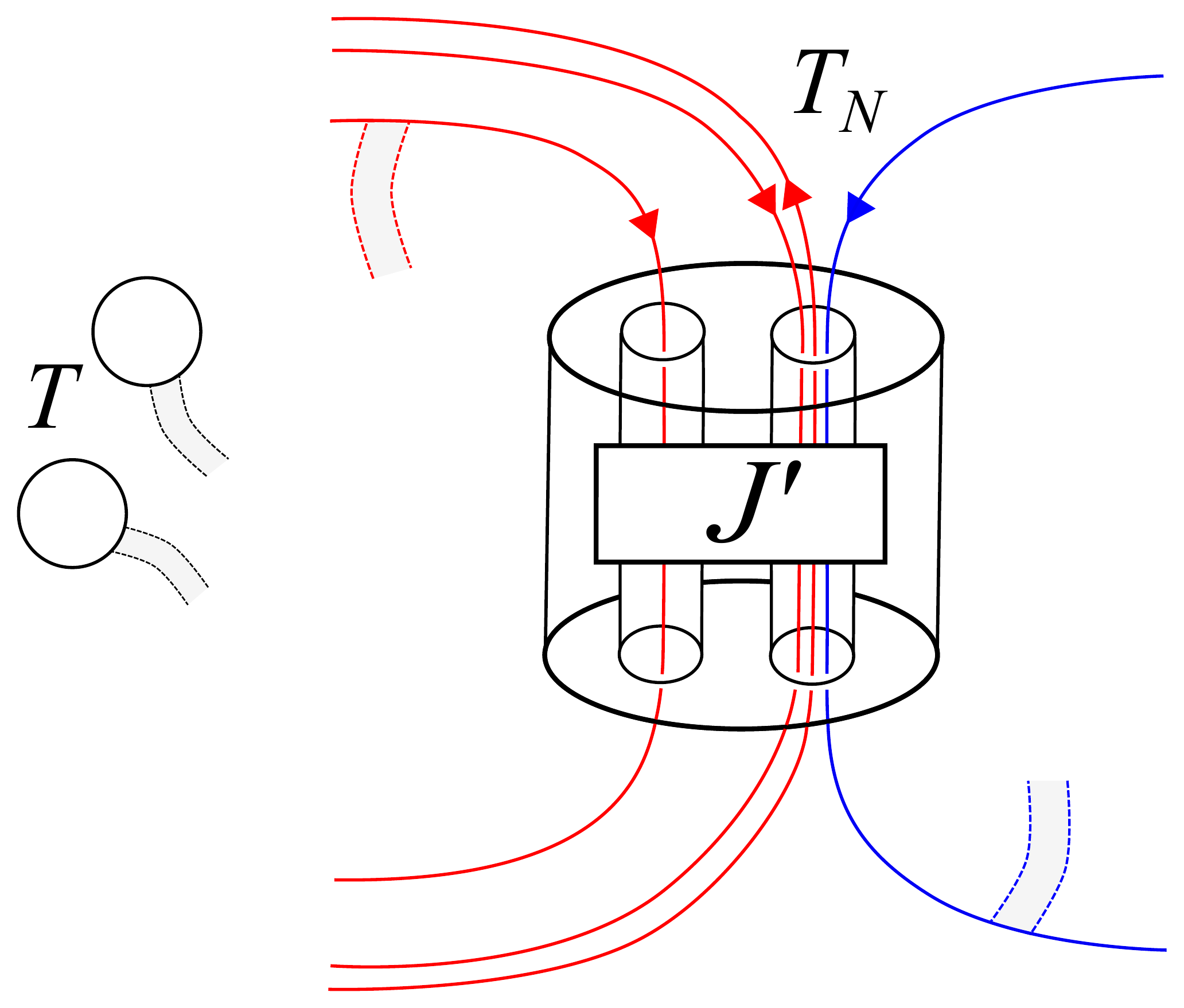}
	\caption{Infected link with fusion bands.}
	\label{fig:morse2}
	\end{center} 
\end{figure}

Remaining is the case where fusion bands cannot be made to avoid intersecting $\mathbb{E}_\varphi\times[0,1]$. Then embed a multidisk $\mathbb{E}'$ with subdisks $D_{1}',...,D_{m}'$ via $\varphi':\mathbb{E}'\hookrightarrow S^3$ such that $\mathbb{E}'_{\varphi'}\times[0,1]$ intersects $M$ so that each thickened subdisk $D_{i}'\times [0,1]$ contains a trivial $n_i$-component string link corresponding to the $i^{th}$ set of components in the $(n_1,...,n_m)$ shaking of $L'$. We may choose this embedding so that it avoids intersecting the fusion bands, $T$, and $\mathbb{E}_\varphi\times[0,1]$.

\begin{figure}[ht]
	\begin{center}
	\includegraphics[width=.4\textwidth]{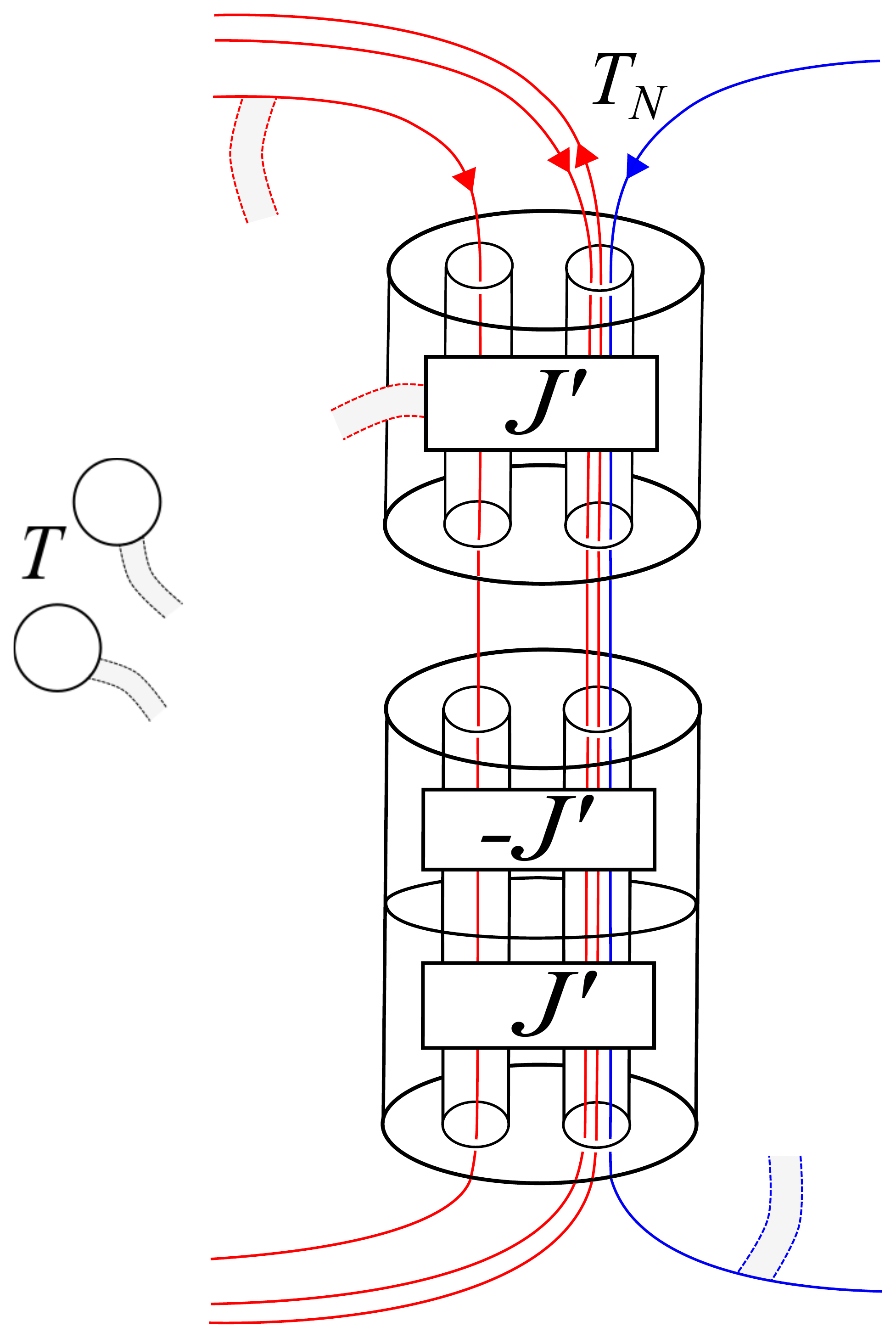}
	\caption{Infecting link to avoid fusion bands.}
	\label{fig:morse3}
	\end{center}
\end{figure}

Now, we may multi-infect $M$ at $\mathbb{E}'_{\varphi'}$ with the string link $J'\#-J'$ and framing $0$; see Figure \ref{fig:morse3}. This is a slice string link, thus $M$ is concordant to the multi-infected link $I(M,J'\#-J',\mathbb{E}'_{\varphi'})$. We can think of multi-infecting along $\mathbb{E}'_{\varphi'}$ by $J'\#-J'$ as infecting along $\mathbb{E}'_{\varphi_1'}$ by $J'$ where $\mathbb{E}'_{\varphi_1'}\times[0,1]:=\mathbb{E}'_{\varphi'}\times[0,0.5]$ and along $\mathbb{E}'_{\varphi_2'}$ by $-J'$ with where $\mathbb{E}'_{\varphi_2'}\times[0,1]:=\mathbb{E}'_{\varphi'}\times[0.5,1]$.


Now define $S=I(M,-J',\mathbb{E}'_{\varphi_2'})$. Then notice $I(S,J',\mathbb{E}'_{\varphi_1'})$ is isotopic to $I(M,J'\#-J',\mathbb{E}'_{\varphi'})$. Observe, $S$ is slice since it is the fusion of the trivial link $T$ and an infection of the trivial link $T_N$ by the slice string link $J'\#-J'$. Hence, $J$ is concordant to $I(S,J',\mathbb{E}'_{\varphi_1'})$, as desired.
\end{proof}

\begin{cor}
If $L$ is shake concordant to $L'$, then  $I(S,J,\mathbb{E}_\varphi)$ is concordant to $I(S',J',\mathbb{E'}_{\varphi'})$ for some $m$-component slice links $S$ and $S'$, $m$-component string links $J$ and $J'$ with closures $\widehat{J}=L$ and $\widehat{J'}=L'$, and embeddings of multidisks $\mathbb{E}_\varphi$ respecting $S$ and $\mathbb{E'}_{\varphi'}$ respecting $S'$.
\label{cor:shakethenconc}
\end{cor}
\begin{proof}
Suppose $L$ is shake concordant to $L'$. Then there exist surfaces $F_1,..,F_m$ that bound a shaking of $L$ in $S^2\times\{0\}$ and a shaking of $L'$ in $S^2\times\{1\}$. As in the proof of Lemma \ref{lemma:shakeimplies}, we may construct a Morse function $f:S^3\times[0,1]\rightarrow[0,1]$ such that when restricted to each $F_i$, all maxima occur at level $\{\frac{4}{5}\}$, split saddles at level $\{\frac{3}{5}\}$, join saddles at level $\{\frac{2}{5}\}$, and local minima at level $\{\frac{1}{5}\}$. Hence, the level $\{\frac{1}{2}\}$ of each $F_i$ is a connected component $M_i$ of some $m$-component link $M$. Thus $M$ is $(1,...,1; n_1,...,n_d)$ shake concordant to $L$ and $M$ is $(1,...,1; s_1,...,s_d)$ shake concordant to $L'$. Therefore, by Lemma \ref{lemma:shakeimplies}, there exists $S$, $S'$, $\mathbb{E}_\varphi$, and $\mathbb{E}'_{\varphi'}$ with the stated conditions such that $M$ is concordant to $I(S,J,\mathbb{E}_\varphi)$ and $M$ is concordant to $I(S',J',\mathbb{E}'_{\varphi'})$.
\end{proof}

Our classification theorem results from combining Corollary \ref{cor:concthenshake} and Corollary \ref{cor:shakethenconc}:
\begin{thm}
\label{thm:classification}
The $m$-component links $L$ and $L'$ are shake concordant if and only if there exist links obtained by string link infection $I(S,J,\mathbb{E}_\varphi)$ and $I(S', J',\mathbb{E}'_{\varphi'})$ that are concordant for some:
		\begin{itemize}
				\item $m$-component slice links $S$ and $S'$,
				\item $m$-component string links $J$, $J'$ with closures $\widehat{J}=L$ and $\widehat{J}'=L'$,
				\item and embedded multidisks $\mathbb{E}_\varphi$ that respects $L$ and $\mathbb{E}'_{\varphi'}$ that respects $L'$.
		\end{itemize}
\label{thm:classification}
\end{thm}

This then gives the following classification of shake slice links.

\begin{cor}
\label{cor:classification}
	The $m$-component link $L$ is shake slice if and only if the link obtained by string link infection $I(S,J,\mathbb{E}_\varphi)$ is slice for some:
		\begin{itemize}
				\item $m$-component slice link $S$,
				\item $m$-component string link $J$ with closure $\widehat{J}=L$,
				\item and embedded multidisk $\mathbb{E}_\varphi$ that respect $S$.
		\end{itemize}
\label{cor:classifyshakeslice}
\end{cor}

\begin{proof}
	A link $L$ is shake slice if and only if it is $(n_1,...,n_m;1,...,1)$ shake concordant to the trivial link. By Theorem \ref{thm:classification} this is equivalent to the links obtained by string link infection $I(S,J,\mathbb{E}_\varphi)$ and $I(S',J',\mathbb{E'}_{\varphi'})$ being concordant where $\widehat{J}=L$ and $\widehat{J}'$ is a trivial link. Hence, $I(s',J',\mathbb{E}_{\varphi'}')$ is slice.
\end{proof}

\subsection*{Remark} In Theorem \ref{thm:classification} {\it shake concordant} can be modified to {\it strong shake concordant} by also modifying the condition on multidisks from {\it respect} to {\it strongly respect}. Similarly, in Corollary \ref{cor:classifyshakeslice}, {\it shake slice} can be modified to {\it strong shake slice} by also modifying the condition on multidisks from {\it respect} to {\it strongly respect}.  The proofs are identical up to maintaining the modified conditions throughout.


\section{Invariance of Milnor Invariants}
In this section, we show that the first non-vanishing Milnor invariant is an invariant of shake concordance. We will need the following lemma.

\begin{lem}
	\label{lemma:milnor}
		Let $I$ be a multi-index whose indices are from $\{1,...,m\}$ and let $k_i$ be the number of occurrences of the index $i$ in $I$. Let $L=L_1\sqcup...\sqcup L_m$ be an $m$-component link with $\overline{\mu}_L(I')=0$ whenever $|I'|<|I|$ and let $J$ be an $m$-component string link whose closure $\widehat{J}$ has $\overline{\mu}_{\widehat{J}}(I')=0$ whenever $|I'|<|I|$. Let embedded $m$-multi-disk $\mathbb{E}_\varphi$ respect $L$. Then $I(L,J,\mathbb{E}_\varphi)$ is also a link with $\overline{\mu}_{I(L,J,\mathbb{E}_\varphi)}(I')=0$ whenever $|I'|<|I|$ and \[\overline{\mu}_{I(L,J,\mathbb{E}_\varphi)}(I)=\overline{\mu}_L(I)+\overline{\mu}_{\widehat{J}}(I).\]
	\end{lem}
	\begin{proof}
		If the embedded multidisk $\mathbb{E}_{\varphi}$ with subdisks $D_1,...,D_m$ strongly respects $L$, then the desired result follows immediately from \cite[Lemma 4.1]{jangkimpowell14}. Otherwise, we do not meet the conditions of \cite[Lemma 4.1]{jangkimpowell14}, but its proof generalizes to accommodate this case, which we now provide for completeness.
		
		Suppose $D_i\cap L_j$ contains $a_{ij}$ positive and $b_{ij}$ negative intersection points. Since $\mathbb{E}_\varphi$ respects $L$, $a_{ii}-b_{ii}=1$ for all $1\leq i\leq m$ and $a_{ij}-b_{ij}=0$ for $i\neq j$. Denote $a_i=\sum_j a_{ij}$ and $b_i=\sum_j b_{ij}$. Let $J'$ be the oriented string link generated by taking $a_i$ parallel copies of the $i$-th component $J_i$ of $J$ and $b_i$ parallel copies of $J_i$ with opposite orientation, for $i=1,...,m$. Notice, $I(L,J,\mathbb{E}_\varphi)$ is the outcome of performing band sums on the split union of $L$ and $\widehat{J'}$ as in Figure \ref{fig:lemma}.
		
		In $\widehat{J'}$ label each parallel copy of $\widehat{J_i}$ with a distinct index $j\in \{1,...,a_i+b_i\}$ for $i=1,...,m$. Define the function $h:\{1,...,\sum_i(a_i+b_i)\}\rightarrow\{1,...,m\}$ which sends the index of a parallel copy of $\widehat{J_i}$ to $j$ where $L_j$ is the component of $L$ that the parallel copy is adjoined to by band sum. By a slight generalization of \cite[Theorem 8.13]{cochran90}, we may determine the contribution to $\overline{\mu}_{I(L,J,\mathbb{E}_\varphi)}(I)$ by summing over all possibilities of $I'$ such that $h(I')=I$:
		\[\overline{\mu}_{I(L,J,\mathbb{E}_\varphi)}(I)=\overline{\mu}_{L}(I)+\sum_{\{I'|h(I')=I\}}\overline{\mu}_{\widehat{J'}}(I').\] 
		
	To calculate the last term, we introduce the function $g:\{1,...,\sum_i(a_i+b_i)\}\rightarrow\{1,...,m\}$ which sends the index of a parallel copy of $\widehat{J_i}$ to $i$. Then:
		\[\sum_{\{I'|h(I')=I\}}\overline{\mu}_{\widehat{J'}}(I')=\sum_{\{I'|h(I')=I,\; g(I')=I\}}\overline{\mu}_{\widehat{J'}}(I')+\sum_{\{I'|h(I')=I,\; g(I')\neq I\}}\overline{\mu}_{\widehat{J'}}(I').\]

	
		Recall that reversing the orientation on a single component $L_i$ of a link $L$ changes the sign of the Milnor invariant $\overline{\mu}_L(I)$ by $(-1)^{k_i}$ where $k_i$ is the number of times $i$ appears in $I$. When $h(I')=I\neq g(I')$, there exists a component labeled $i$ such that $h(i)\neq g(i)$. But as $\mathbb{E}_\varphi$ respects $L$, the component labeled $i$ forms a pair with a component labeled $j$ with opposite orientation such that $i$ and $j$ are parallel copies of $\widehat{J}_{g(i)}$ and are adjoined by band sum to $L_{h(i)}$. That is, $g(i)=g(j)$ and $h(i)=h(j)$. If $I'_{i\rightarrow j}$ denotes the multi-index formed by replacing each $i$ in $I'$ with $j$, then we have:
		\[\sum_{\{I'|h(I')=I,\; g(I')\neq I\}}\overline{\mu}_{\widehat{J'}}(I')=\sum_{\{I'_{i\rightarrow j}|h(I'_{i\rightarrow j})=I,\; g(I'_{i\rightarrow j})\neq I\}}\overline{\mu}_{\widehat{J'}}(I'_{i\rightarrow j}).\]	
		But since $\overline{\mu}_{\widehat{J'}}(I')=-\overline{\mu}_{\widehat{J'}}(I'_{i\rightarrow j})$, it follows:
		\[\sum_{\{I'|h(I')=I\}}\overline{\mu}_{\widehat{J'}}(I')=\sum_{\{I'|h(I')=I,\; g(I')=I\}}\overline{\mu}_{\widehat{J'}}(I')=\sum_{\{I'|h(I')=I,\; g(I')=I\}}\overline{\mu}_{\widehat{J}}(I)\cdot\prod_{j\in I'} r_j^{\lambda_j}\]
	where $r_j\in \{\pm1\}$ is $1$ if the parallel copy of a component of $\widehat{J}$ with index $j$ has the same orientation as $J_{g(j)}$ and $-1$ otherwise, and $\lambda_j$ denotes the number of times that $j$ appears in $I'$. To count the number of $I'$ such that $g(I')=h(I')=I$, notice there are $k_i$ choices  for each $i\in I$ with a signed contribution of $(a_{ii}-b_{ii})^{k_i}$. Therefore,
	\[\sum_{\{I'|h(I')=I,\; g(I')=I\}}\prod_{j\in I'} r_j^{\lambda_j}=\prod_{i\in I} (a_{ii}-b_{ii})^{k_i}=1\]
	and hence,
	\[\overline{\mu}_{I(L,J,\mathbb{E}_\varphi)}(I)=\overline{\mu}_{L}(I)+\overline{\mu}_{\widehat{J}}(I).\]
	\end{proof}

We can now prove that the first non-vanishing Milnor invariants are preserved under shake concordance.

\begin{thm}
\label{thm:firstmilnor}
	If two links $L$ and $L'$ are shake concordant, then they have equal first non-vanishing Milnor invariants. That is, if for some multi-index $I$, $\overline{\mu}_L(I)\neq 0$ and $\overline{\mu}_L(J)=0$ for all $|J|<|I|$, then $\overline{\mu}_L(I)=\overline{\mu}_{L'}(I)$ and $\overline{\mu}_{L'}(J)=0$ for all $|J|<|I|$.	
\end{thm}

\begin{proof}
	Suppose $J$ is shake concordant to $J'$. Then we can find slice links $L$ and $L'$, string links $s$ and $s'$ such that $\widehat{s}=J$ and $\widehat{s'}=J'$, and multidisks $\mathbb{E}_\varphi$ and $\mathbb{E'}_{\varphi'}$ respecting $L$ and $L'$, respectively, such that $I(L,s,\mathbb{E}_\varphi)$ is concordant to $I(L',s',\mathbb{E'}_{\varphi'})$. Hence for any multi-index $I$,
	\[\overline{\mu}_{I(L,s,\mathbb{E}_\varphi)}(I)=\overline{\mu}_{I(L',s',\mathbb{E'}_{\varphi'})}(I).\]

	Moreover, by Lemma \ref{lemma:milnor} we have
	\[\overline{\mu}_{I(L,s,\mathbb{E}_\varphi)}(I)=\overline{\mu}_{L}(I)+\overline{\mu}_{J}(I), \;\;\; \overline{\mu}_{I(L',s',\mathbb{E'}_{\varphi'})}(I)=\overline{\mu}_{L'}(I)+\overline{\mu}_{J'}(I).\]
	Since $L$ and $L'$ are slice, all of their Milnor invariants vanish, hence we have,
	\[\overline{\mu}_{J}(I)=\overline{\mu}_{I(L,s,\mathbb{E}_\varphi)}(I)=\overline{\mu}_{I(L',s',\mathbb{E}_{\varphi'})}(I)=\overline{\mu}_{J'}(I).\]
\end{proof}

\begin{cor}
\label{cor:linkingpreserved}
	Linking number is an invariant of shake concordance. That is, if $L$ and $L'$ are shake concordant, then $lk(L_i,L_j)=lk(L_i',L_j')$ where $i\neq j$.
\end{cor}

\begin{figure}[!h]
	\begin{center}
	\subfloat[]{\includegraphics[width=.25\textwidth]{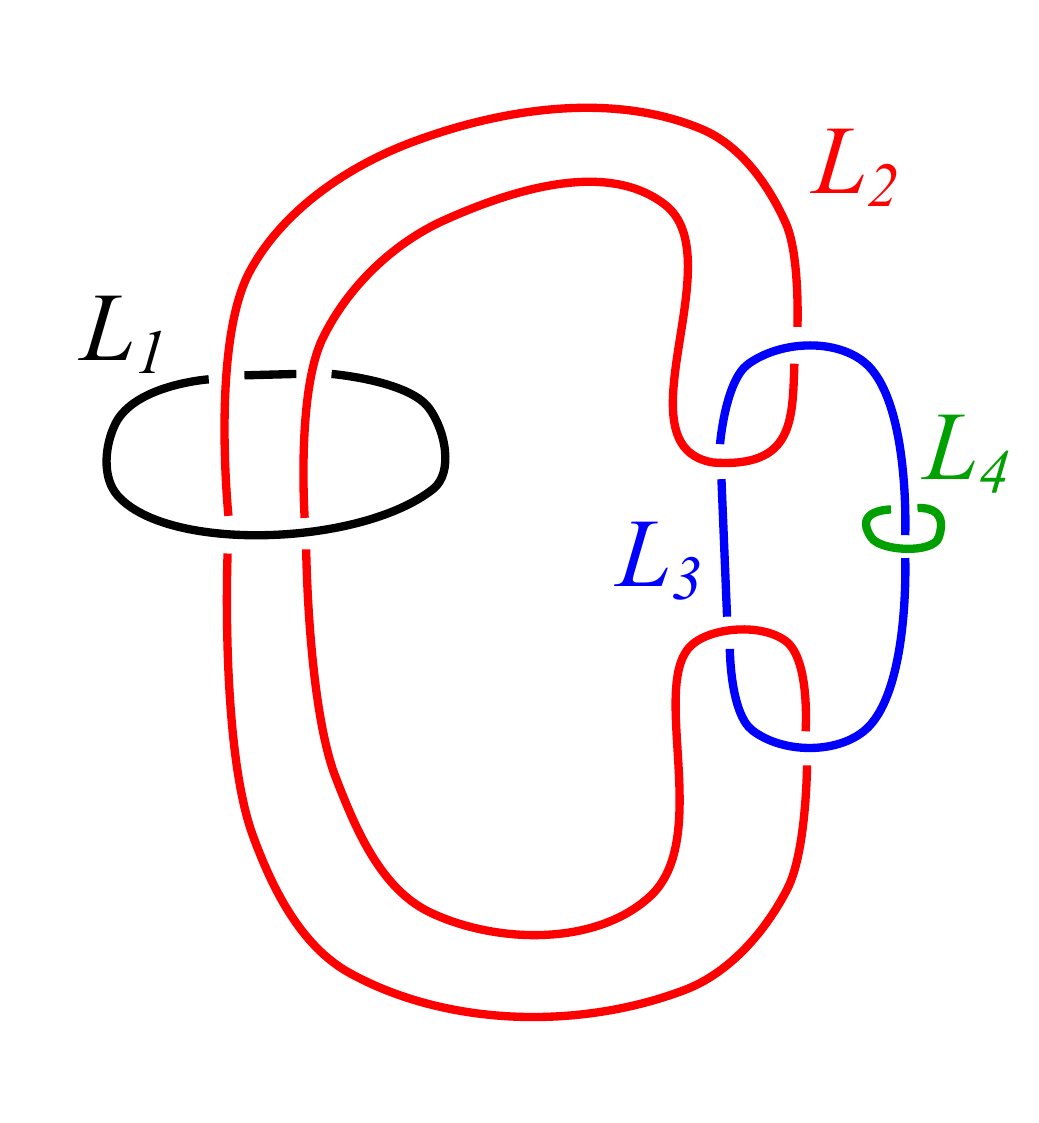}}
	\subfloat[]{\includegraphics[width=.25\textwidth]{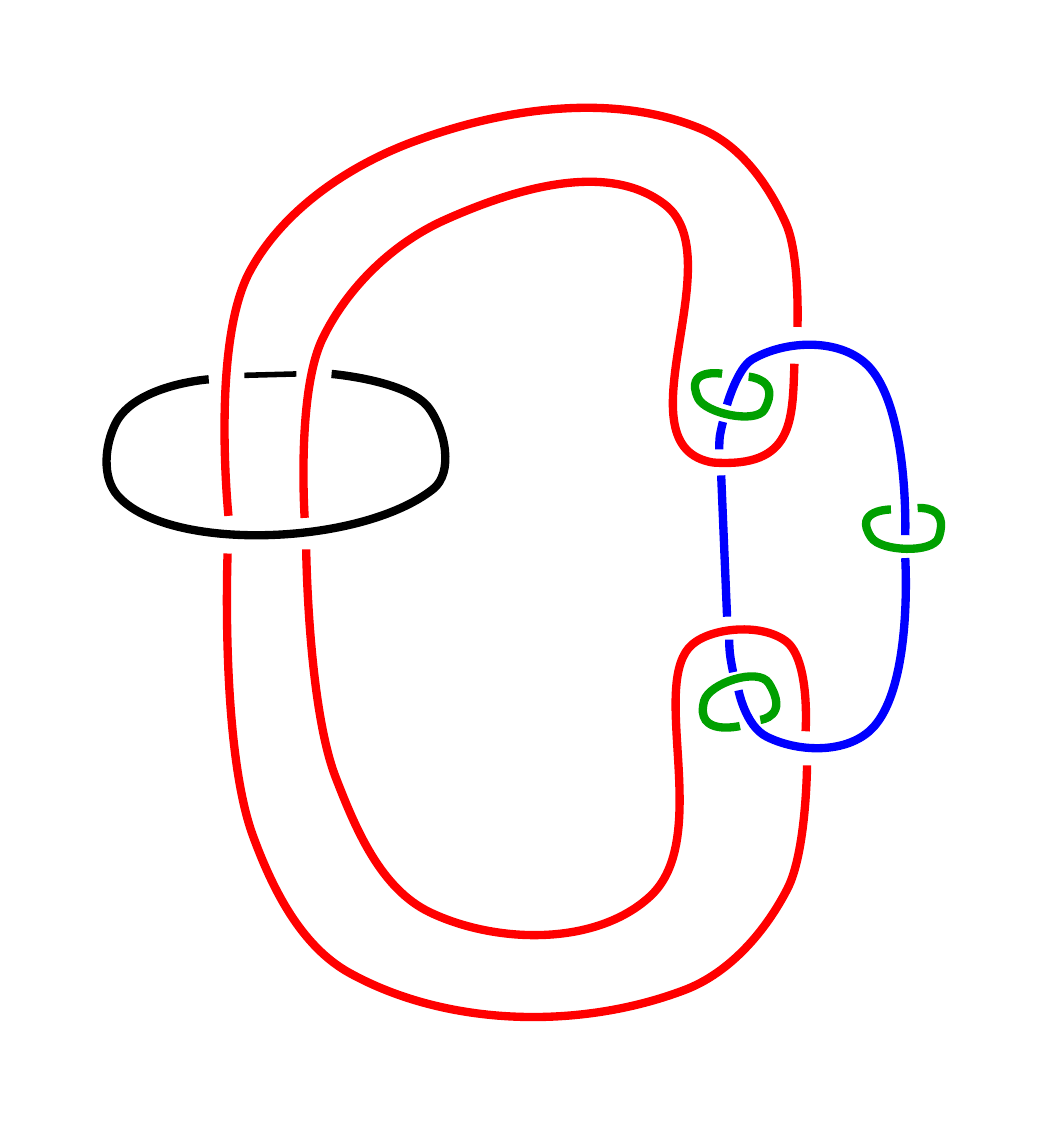}}
	\subfloat[]{\includegraphics[width=.25\textwidth]{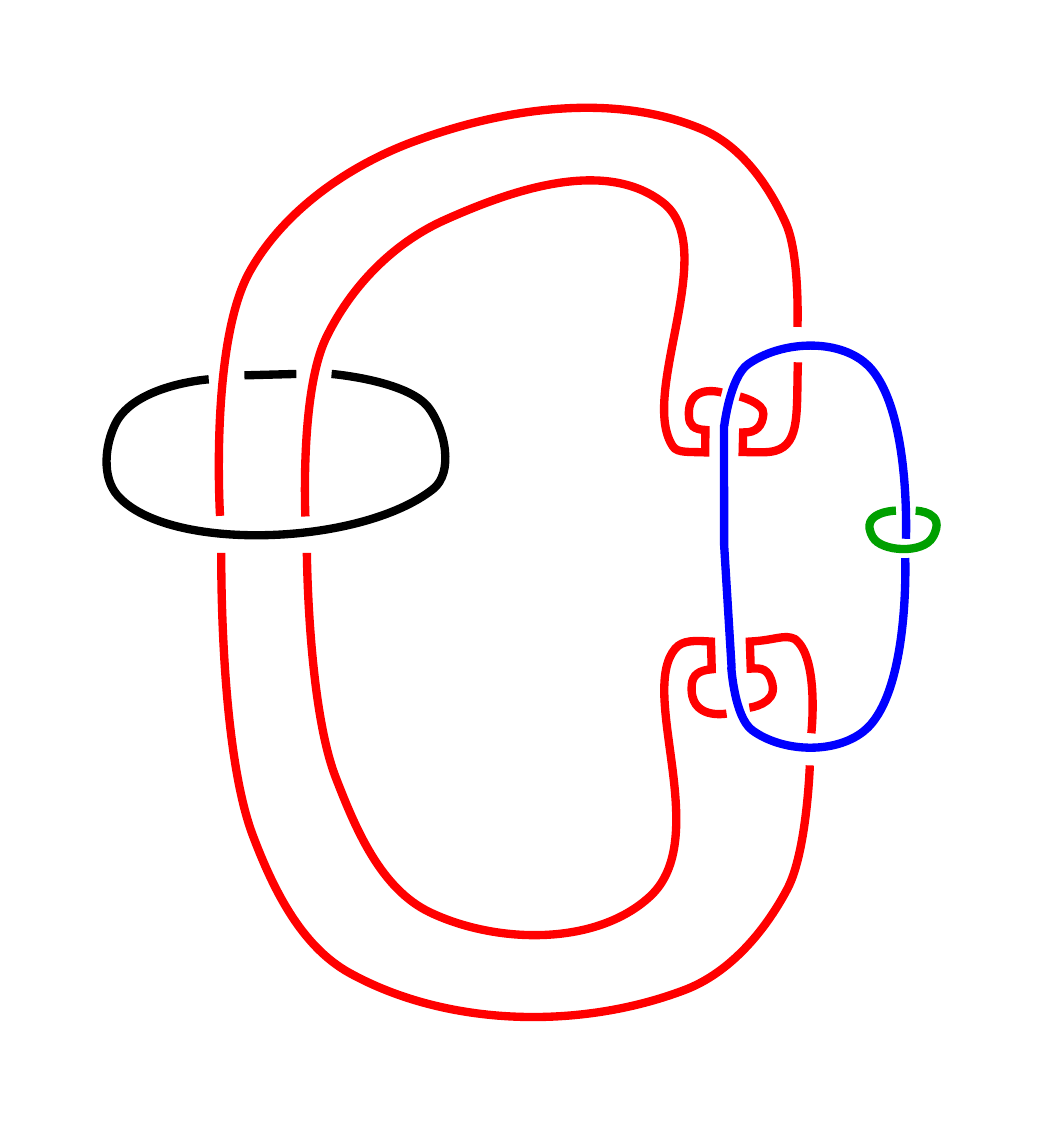}}
	\subfloat[]{\includegraphics[width=.25\textwidth]{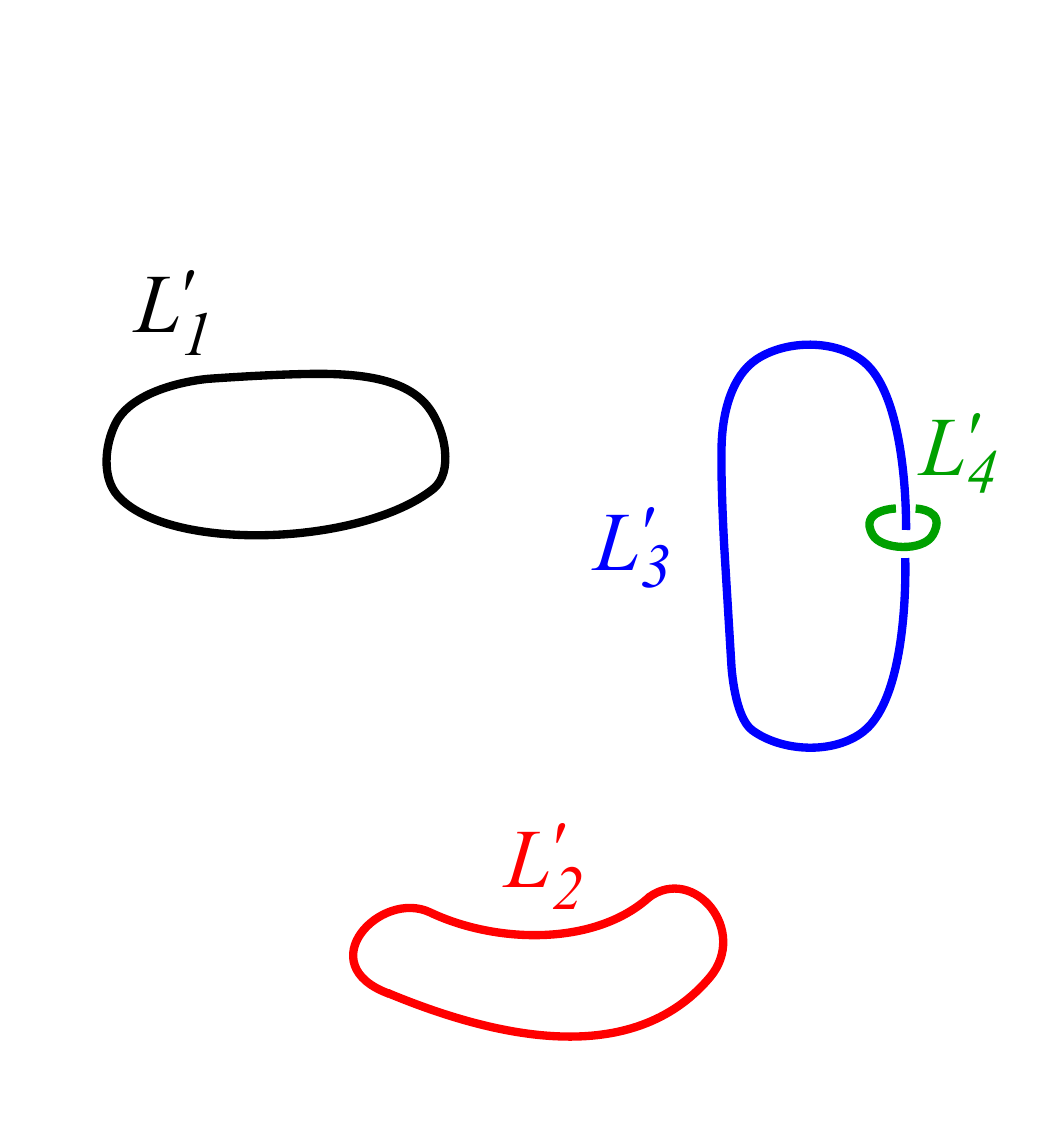}}
	\caption{Shake concordance between (A) and (D) with $\overline{\mu}_L(123)\neq \overline{\mu}_{L'}(123)$.}
	\label{fig:4thbring}
	\end{center}
\end{figure}
	
In general, though, not all Milnor invariants are preserved by shake concordance. For instance, consider the links $L$ in Figure \ref{fig:4thbring} (A)  and $L'$ in Figure \ref{fig:4thbring} (D). $L$ is shake concordant to $L'$. We can see this by taking a 3-component shaking of the component $L_4$ and adjoining  two of the parallel copies via band sum to $L_2$ as shown in Figure \ref{fig:4thbring} (B) and (C).

Notice, however, that the sublink $S=L_1\sqcup L_2\sqcup L_3$ of $L$ is the Borromean rings and hence we have $\overline{\mu}_{L}(123)=\overline{\mu}_{S}(123)=\pm1$, where the sign is determined by the orientations of the components. Meanwhile, the sublink $S'=L_1'\sqcup L_2'\sqcup L_3'$ of $L'$ is a trivial link and hence $\overline{\mu}_{L'}(123)=\overline{\mu}_{S'}(123)=0$.

Link concordance implies link homotopy \cite{giffen} \cite{goldsmith}. It is therefore natural to ask if shake concordance implies link homotopy. Milnor showed that a link is null homotopic if and only if all its $\overline{\mu}$ invariants with non-repeating invariants vanish \cite{milnor54}. Hence, if a link is shake concordant to the trivial link, the by Theorem \ref{thm:firstmilnor} all its $\overline{\mu}$ invariants vanish and hence it is it is null homotopic. Moreover, since 2-component links are classified up to link homotopy by linking number, it follows from Corollary \ref{cor:linkingpreserved} that if two 2-component links are shake concordant then they are link homotopic. However, in general shake concordance does not imply link homotopy. For instance, links $L$ in Figure \ref{fig:4thbring} (A) and $L'$ in Figure \ref{fig:4thbring} (D) are shake concordant but have differing $\bar{\mu}(123)$ and hence are not link concordant.

There also exists links that are strongly shake concordant but not link homotopic. For instance, consider links $L$ in Figure \ref{fig:strshconc} (A) and $L'$ in Figure \ref{fig:strshconc} (C). These links are strongly shake concordant and have identical $\overline{\mu}$ invariants. However, Figure \ref{fig:strshconc} shows that $L$ is ambient isotopic to the link in Figure \ref{fig:levine} (H) which Levine demonstrated is not link homotopic to $L'$  \cite[Section 13]{levine}.
\bibliographystyle{spmpsci}

\begin{figure}[!h]
	\begin{center}
	\subfloat[]{\includegraphics[width=.2\textwidth]{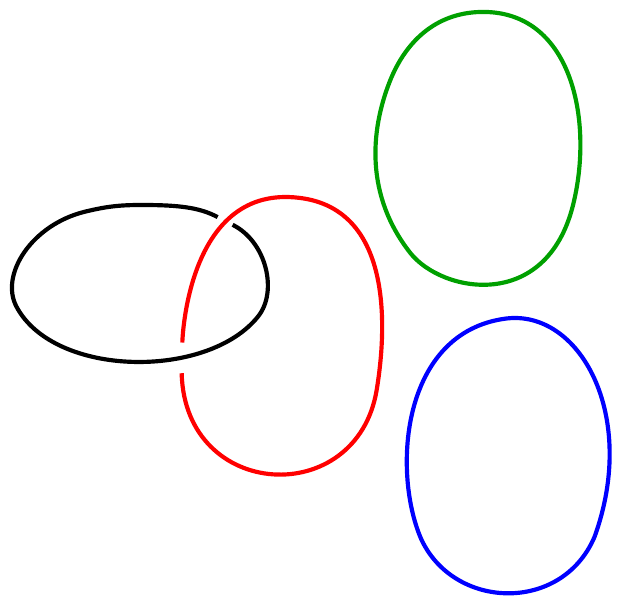}}\hspace{0.5in}
	\subfloat[]{\includegraphics[width=.2\textwidth]{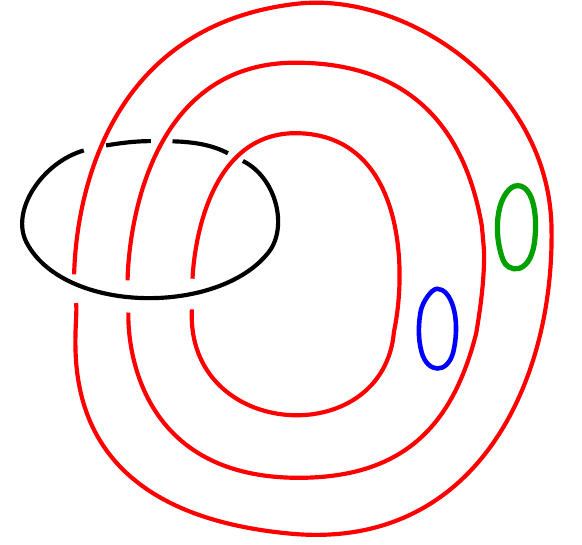}}\hspace{0.5in}
	\subfloat[]{\includegraphics[width=.2\textwidth]{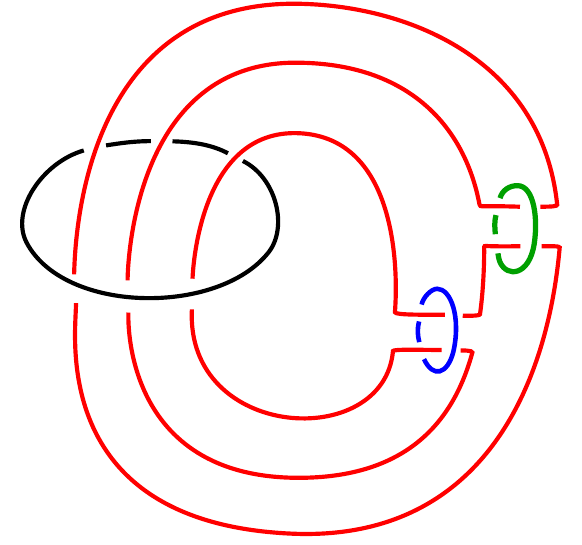}}
	\caption{A strong shake concordance between (A) and (C) as the link in (A) has shaking (B) which cobounds the genus zero surfaces of a strong shake concordance with (C).}
	\label{fig:strshconc}
	\end{center}
\end{figure}

\begin{figure}[!h]
	\begin{center}
	\subfloat[]{\includegraphics[width=.2\textwidth]{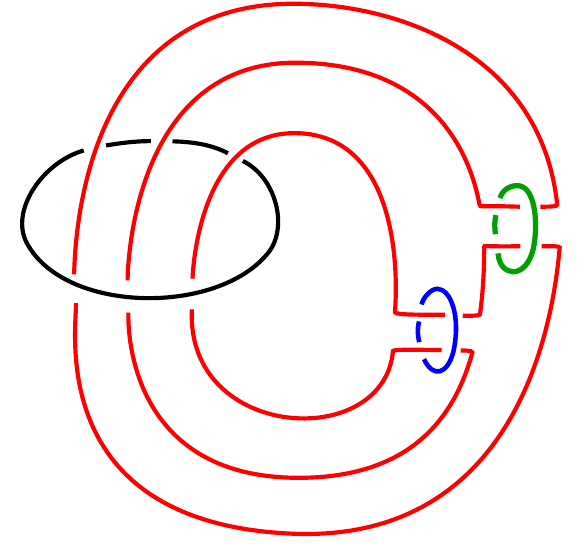}}\hspace{0.5in}
	\subfloat[]{\includegraphics[width=.2\textwidth]{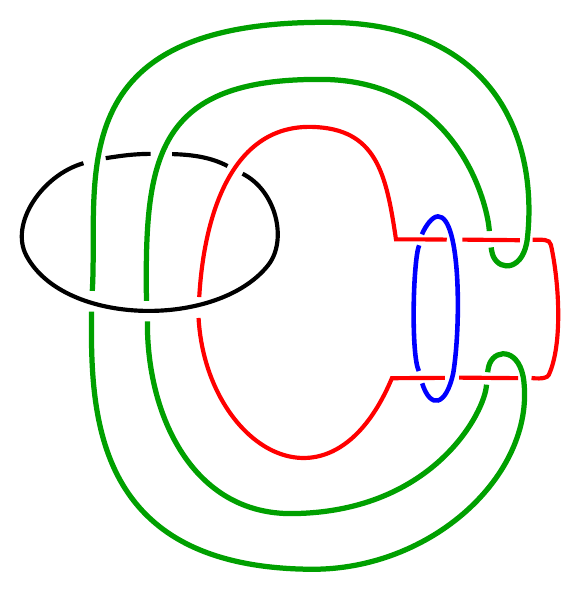}}\hspace{0.5in}
	\subfloat[]{\includegraphics[width=.2\textwidth]{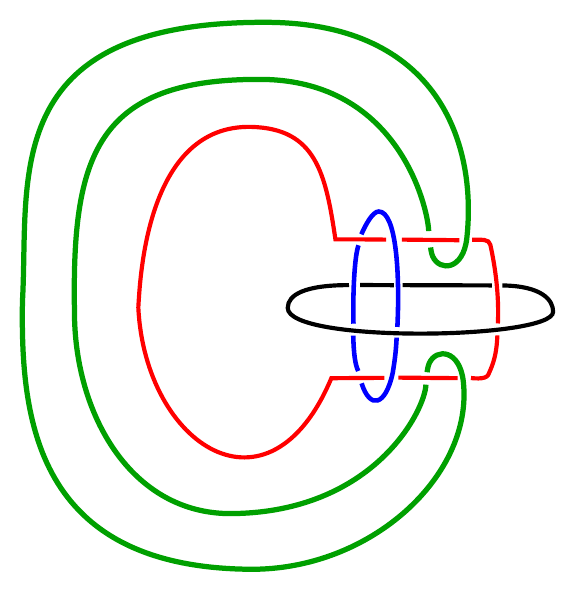}}\\
	
	\subfloat[]{\includegraphics[width=.2\textwidth]{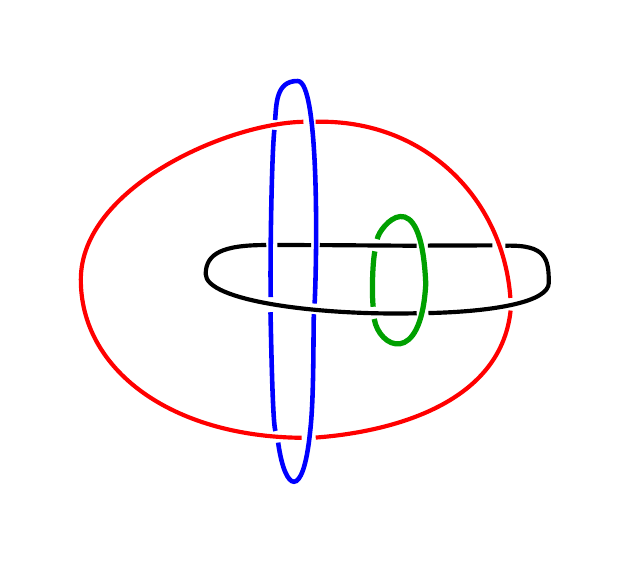}}\hspace{0.55in}
	\subfloat[]{\includegraphics[width=.2\textwidth]{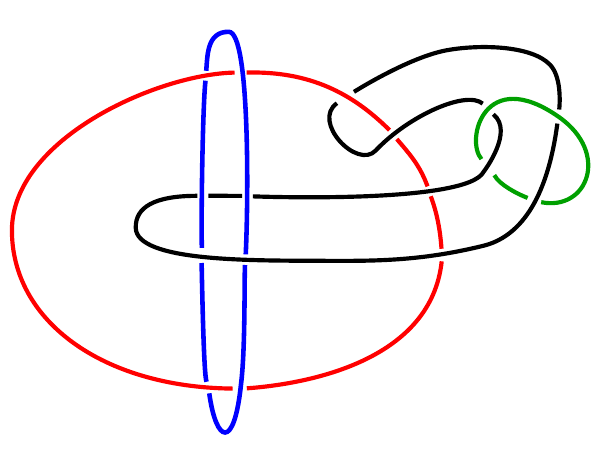}}\hspace{0.55in}
	\subfloat[]{\includegraphics[width=.2\textwidth]{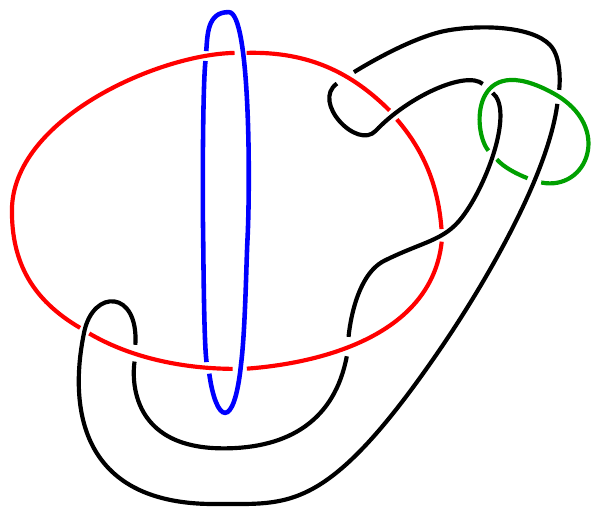}}\\
	
	\hspace{0.4in}
	\subfloat[]{\includegraphics[width=.2\textwidth]{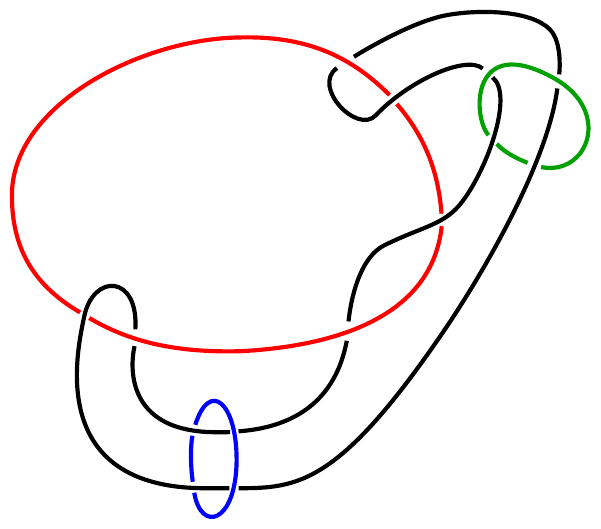}}\hspace{0.3in}
	\subfloat[]{\includegraphics[width=.3\textwidth]{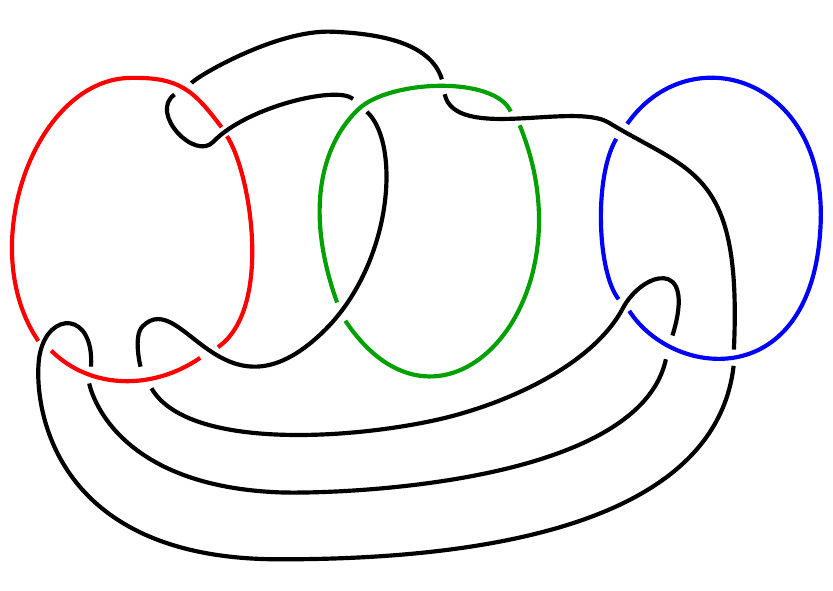}}
	\caption{An ambient isotopy from the link in (A) to the link in (H).}
	\label{fig:levine}
	\end{center}
\end{figure}

\pagebreak

\bibliographystyle{alpha}
\bibliography{bib}

\end{document}